\documentclass[12pt,twoside]{amsart}
\usepackage{amsfonts}
\usepackage{amssymb}
\usepackage{amsthm}
\usepackage{amsmath}
\usepackage{newlfont}
\usepackage{t1enc}
\usepackage{bbm}
\usepackage{fontsmpl}
\usepackage{graphics}
\usepackage{graphicx}
\usepackage{color}
\usepackage{psfrag}
\usepackage{a4wide}
\usepackage{amssymb}
\usepackage{enumerate}
\usepackage{hyperref}

\theoremstyle{remark}

\theoremstyle{definition}
\newtheorem{teo}{Theorem}[section]

\newtheorem{lema}[teo]{Lemma}
\newtheorem{prop}[teo]{Proposition}
\newtheorem{cor}[teo]{Corollary}
\newtheorem{obs}[teo]{Remark}
\newtheorem{defi}[teo]{Definition}

\newtheorem{cond}[teo]{Conditions}

\numberwithin{equation}{section}
\theoremstyle{plain}
\newtheorem{theorem}{Theorem}[section]

\theoremstyle{definition}

\newcommand{\ve}{\varepsilon}
\newcommand{\R}{{\mathbb R}}

\newcommand{\Z}{{\mathbb{Z}}}

\newcommand{\N}{{\mathbb N}}

\newcommand{\F}{{\mathcal F}}
\newcommand{\G}{{\mathcal G}}

\newcommand{\B}{{\mathcal B}}

\newcommand{\0}{{\mathbf 0}}

\def\F{\mathcal F}

\newcommand{\p}{\partial}

\def\1{{\mathbf 1}}

\def\R{\mathbb{R}}
\def\N{\mathbb{N}}

\def\F{\mathcal F}

\numberwithin{equation}{section}

\begin{document}

\title{Metastability for small random perturbations of a PDE with blow-up}

\author{Pablo Groisman, Santiago Saglietti and Nicolas Saintier}

\thanks{2010 {\it Mathematics Subject Classification:} 60H15, 35K57.}

\keywords{stochastic partial differential equations, random perturbations, blow-up, metastability}

\address{Departamento de Matem\'atica\hfill\break
\indent Facultad de Ciencias Exactas y Naturales\hfill\break
\indent Universidad de Buenos Aires \hfill\break
\indent Pabell\'{o}n I, Ciudad Universitaria \hfill\break
\indent C1428EGA Buenos Aires, Argentina.
}
\email{{\tt pgroisma@dm.uba.ar, ssaglie@dm.uba.ar, nsaintie@dm.uba.ar}
\hfill\break\indent {\it
Web page:} {\tt http://mate.dm.uba.ar/$\sim$pgroisma}}

\date{}

\begin{abstract}
We study small random perturbations by additive space-time white noise of a reaction-diffusion equation with a unique stable equilibrium and solutions
which blow up in finite time.
We show that for initial data in the domain of attraction of the stable equilibrium the perturbed system exhibits metastable behavior: its time averages
remain stable around this equilibrium until an abrupt and unpredictable
transition occurs which leads to explosion in a finite (but exponentially large)
time. On the other hand, for initial data in the domain of explosion we show
that the explosion time of the perturbed system converges to the explosion time
of the deterministic solution.
\end{abstract}

\maketitle

\section{Introduction}

We consider, for $\varepsilon > 0$, the stochastic process $U^{u,\varepsilon}$ which formally satisfies the stochastic partial differential equation
\begin{equation}\label{MainSPDE}
\left\{\begin{array}{rll}
\p_t U^{u,\varepsilon} &= \p^2_{xx}U^{u,\varepsilon} + g(U^{u,\varepsilon}) + \varepsilon \dot{W}& \quad t>0 \,,\, 0<x<1 \\
U^{u,\varepsilon}(t,0)&=U^{u,\varepsilon}(t,1)=0 & \quad t>0 \\
U^{u,\varepsilon}(0,x) &=u(x)
\end{array}\right.
\end{equation} where $g:\R \rightarrow \R$ is given by $g(u):=u|u|^{p-1}$ for fixed $p > 1$, $\dot{W}$ is space-time white noise and $u$ is a continuous function satisfying $u(0)=u(1)=0$.

This process can be thought of as a random perturbation of the dynamical system $U^u$ given by the
solution of \eqref{MainSPDE} with $\ve=0$, i.e. $U^u$ satisfies the partial differential equation
\begin{equation}
\label{MainPDE}
\left\{\begin{array}{rll}
\p_t U^u &= \p^2_{xx}U^u + g(U^u) & \quad t>0 \,,\, 0<x<1 \\
U^u(t,0)& =0 & \quad t>0 \\
U^u(t,1) & = 0 & \quad t>0 \\
U^u(0,x) &=u(x) & \quad 0<x<1.
\end{array}\right.
\end{equation}
Equation \eqref{MainPDE} is of reaction-diffusion type, a broad class of evolution equations which naturally arise
in the study of phenomena as diverse as diffusion of a fluid through a porous material,
transport in a semiconductor, coupled chemical reactions with spatial diffusion, population genetics, among others. In all these cases, the equation represents an approximate model
of the phenomenon and thus it is of interest to understand how its description might change
if subject to small random perturbations.

An important feature of \eqref{MainPDE} is that it admits solutions which are
only local in time and blow up in a finite time. Indeed, the system has a unique
stable equilibrium, the null function $\mathbf{0}$, and a countable family of
unstable equilibriums, all of which are saddle points. The stable equilibrium
possesses a domain of attraction $\mathcal{D}_{\mathbf{0}}$ such that if $u \in
\mathcal{D}_{\mathbf{0}}$ then the solution $U^u$ of \eqref{MainPDE} with
initial datum $u$ is globally defined and converges to $\mathbf{0}$ as time
tends to infinity. Similarly, each unstable equilibrium has its own stable
manifold, the union of which constitutes the boundary of
$\mathcal{D}_{\mathbf{0}}$. Finally, for $u \in
\mathcal{D}_e:=\overline{\mathcal{D}_{\mathbf{0}}}^c$ the system blows up in
finite time, i.e. there exists a time $0 <\tau^u < +\infty$ such that the solution $U^u$
is defined for all $t \in [0,\tau^u)$ but satisfies
\[
 \lim_{t \nearrow \tau^u}\|U^u(t,\cdot)\|_\infty=+\infty.
\]
The behavior of the system is, in some aspects, similar to the
double-well potential model studied in \cite{B1,MOS}. Indeed,
\eqref{MainPDE} can be reformulated as
$$
\p_t U^u = - \frac{\partial S}{\partial \varphi} (U^u)
$$ where $S$ is the potential formally given by
$$
S(v) = \int_0^1 \left[\frac{1}{2} \left(\frac{dv}{dx}\right)^2 + G(v)\right],
$$ where we take $G(v):=- \frac{|v|^{p+1}}{p+1}$ as opposed to the term $G(v) = \frac{\lambda}{4}v^4 - \frac{\mu}{2}v^2$ appearing in the double-well potential model.
In our system, instead of having two wells, each being the domain of attraction of the two stable
equilibriums of the system, we have only one which corresponds to
$\mathcal{D}_{\mathbf{0}}$. Since our potential tends to $-\infty$ along every
direction, we can imagine the second well in our case as being infinity and thus
there is no return from there once the system reaches its bottom. Moreover,
since the potential behaves like $-\lambda^{p+1}$ in every direction, if the
system falls into this ``infinite well'' it will reach its bottom (infinity) in
a finite time (blow-up).

Upon adding a small noise to \eqref{MainPDE}, one wonders if there are any
qualitative differences in behavior between the deterministic system
\eqref{MainPDE} and its \mbox{stochastic perturbation \eqref{MainSPDE}.} For
short times both systems should behave similarly, since in this case the noise
term will be typically of much smaller order than the remaining terms in the
right hand side of \eqref{MainSPDE}. However, due to the independent and
normally distributed increments of the perturbation, when given enough time the
noise term will eventually reach sufficiently large values so as to induce a
significant change of behavior in \eqref{MainSPDE}. We are interested in
understanding what changes might occur in the blow-up phenomenon due to this
situation and, more precisely, which are the asymptotic properties as
$\varepsilon \rightarrow 0$ of the explosion time of \eqref{MainSPDE} for the
different initial data.
Based on all of the considerations above, we expect the following scenario:
\begin{enumerate}
\item [i.] \textit{Thermalization}. For initial data in
$\mathcal{D}_{\mathbf{0}}$, the stochastic system is at first attracted towards
this equilibrium. Once near it, the terms in the right hand side of
\eqref{MainPDE} become negligible and so the process is then pushed away from
the equilibrium by noise. Being away from $\mathbf{0}$, the noise becomes
overpowered by the remaining terms in the right hand side of \eqref{MainSPDE}
and this allows for the previous pattern to repeat itself: a large number of
attempts to escape from the \mbox{equilibrium,} followed by a strong attraction
towards it.
 \item [ii.] \textit{Tunneling}. Eventually, after many frustrated attempts, the
process succeeds in escaping $\mathcal{D}_{\mathbf{0}}$ and reaches the domain
of explosion, the set of initial data for which \eqref{MainPDE} blows up in
finite time. Since the probability of such an event is very small, we expect
this escape time to be exponentially large.
Furthermore, due to the large number of attempts that are necessary, we also
expect this time to show little memory of the initial data.
\item [iii.] \textit{Final excursion}. Once inside the domain of explosion, the
stochastic system is forced to explode by the dominating source term $g$.
\end{enumerate}
This type of phenomenon is known as \textit{metastability}: the system behaves
for a long time as if it were under equilibrium, but then performs an abrupt
transition towards the real equilibrium (in our case, towards infinity). The
former description was proved rigorously for the (infinite-dimensional)
double-well potential model  in \cite{B1,MOS}, inspired by the work in
\cite{GOV} for its finite-dimensional analogue. Their proofs rely heavily on
large deviations estimates for $U^{u,\varepsilon}$ established in \cite{FJL} for
the infinite-dimensional system and in \cite{FW} for the finite-dimensional
setting. In our case, we are only capable of proving
the existence of local solutions of \eqref{MainSPDE} and in fact,
explosions \textit{will} occur for $U^{u,\varepsilon}$. As a consequence, we
will not be able to apply these same estimates directly, as the validity of
these estimates relies on a proper control of the growth of solutions which does
not hold in our setting. Localization techniques apply reasonably well to deal
with the process
until it escapes any fixed bounded domain but they cannot be used to say what
happens from then onwards. Since we wish to focus specifically on trajectories
that blow up in finite time, it is clear that a new approach is needed for this
last part, one which involves a careful study of the blow-up phenomenon.
Unfortunately, when dealing with perturbations of differential equations with
blow-up, understanding
how the behavior of the blow-up time is modified or even showing the persistence
of the blow-up phenomenon itself is by no means an easy task in most cases.
There are no general results addressing this matter, not even for nonrandom
perturbations. This is why the usual approach to this kind of problems is to
consider particular models such as ours.

The article is organized as follows. In Section \ref{defr} we give some
preliminary definitions, introduce the local Freidlin-Wentzell estimates and
then
detail the results of this article. In Section \ref{deter} we give a detailed
description of the deterministic system \eqref{MainPDE}. Section \ref{explo}
focuses on the explosion time of the stochastic system for initial data in the
domain of explosion. The construction of an auxiliary domain $G$ is performed in Section \ref{secg} and the study of
the escape from $G$ is carried out in Section \ref{secescapedeg}. In Section
\ref{secfinal} we establish metastable behavior for solutions with initial data
in the domain of attraction of the stable equilibrium. Finally, we include at
the end an appendix with some auxiliary results to be used throughout our
analysis.

\section{Definitions and results}\label{defr}

\subsection{The deterministic PDE}

Our purpose in this section is to study equation \eqref{MainPDE}. We assume that the source term $g : \R \rightarrow \R$
is given by $g(u)=u|u|^{p-1}$ for fixed $p > 1$ and also that $u$ belongs to the space $C_{D}([0,1])$ of
continuous functions on $[0,1]$ satisfying homogeneous Dirichlet boundary conditions, namely
$$
C_{D}([0,1])= \{ v \in C([0,1]) : v(0)=v(1)=0 \}.
$$ The space $C_{D}([0,1])$ is endowed with the supremum norm, i.e.
$$
\| v \|_\infty = \sup_{x \in [0,1]} |v(x)|.
$$ For any choice of $r > 0$ and $v \in C_D([0,1])$, we let $B_r(v)$ denote the
closed ball in $C_D([0,1])$ of center $v$ and radius $r$. Whenever the center is
the null function $\mathbf{0}$, we simply write $B_r$.
Equation \eqref{MainPDE} can be reformulated as
\begin{equation}\label{formalPDE}
\p_t U = - \frac{\partial S}{\partial \varphi} (U)
\end{equation} where the \textit{potential} $S$ is the functional on
$C_D([0,1])$ given by
$$
S(v) = \left\{ \begin{array}{ll} \displaystyle{\int_0^1 \left[\frac{1}{2}
\left(\frac{dv}{dx}\right)^2 - \frac{|v|^{p+1}}{p+1}\right]}
              & \text{ if $v \in H^1_0((0,1))$}
                   \\ \\ +\infty & \text{ otherwise.}\end{array}\right.
$$ Here $H^1_0((0,1))$ denotes the Sobolev space of square-integrable functions
defined on $[0,1]$ with square-integrable weak derivative which vanish at the
boundary $\{0,1\}$. Recall that $H^1_0((0,1))$ can be embedded into $C_D([0,1])$
so that the potential is indeed well-defined.  We refer the reader to the
appendix for a review of some of the main properties of $S$ which shall be
required throughout our work.

The formulation on \eqref{formalPDE} is interpreted as the validity of
$$
\int_0^1 \p_t U(t,x) \varphi(x)dx = \lim_{h \rightarrow 0} \frac{S(U + h\varphi)
- S(U)}{h}
$$ for any $\varphi \in C^1([0,1])$ with $\varphi(0)=\varphi(1)=0$. It is known
that for any $u \in C_D([0,1])$ there exists a unique solution $U^{u}$ to
equation \eqref{MainPDE} defined on some maximal time interval $[0,\tau^{u})$
where $0 < \tau^{u} \leq +\infty$ is called the \textit{explosion time} of $U^u$
(see \cite{QS} for further details). In general, we will say that this solution belongs to the
space
$$
C_D([0,\tau^u) \times [0,1]) = \{ v \in C( [0,\tau^{u}) \times [0,1]) :
v(\cdot,0)=v(\cdot,1) \equiv 0 \}.
$$ However, whenever we wish to make its initial datum $u$ explicit we will do
so by saying that the solution belongs to the space
$$
C_{D_{u}}([0,\tau^{u}) \times [0,1]) = \{ v \in C( [0,\tau^{u}) \times [0,1]) :
v(0,\cdot)=u \text{ and }v(\cdot,0)=v(\cdot,1) \equiv 0 \}.
$$ The origin $\mathbf{0} \in C_D([0,1])$ is the unique stable equilibrium of
the system and is in fact asymptotically stable. It corresponds to the unique
local minimum of the potential $S$. There is also a family of unstable
equilibria of the system corresponding to the remaining critical points of the
potential $S$, all of which are saddle points. Among these unstable equilibria
there exists only one of them which is nonnegative, which we shall denote by $z$. It can
be shown that this equilibrium $z$ is in fact strictly positive for $x \in
(0,1)$, symmetric with respect to the axis $x=\frac{1}{2}$ (i.e. $z(x)=z(1-x)$
for every $x \in [0,1]$) and that is both of minimal potential and minimal norm
among all the unstable equilibria. The remaining equilibria are obtained by
alternating scaled copies of both $z$ and $-z$ as Figure \ref{fig0}
shows. We establish this fact rigurously in Section \ref{deter}.

\begin{figure}
$$
	\begin{array}{ccccc}
	\begin{array}{c}\includegraphics[width=4cm,height=1.3cm]{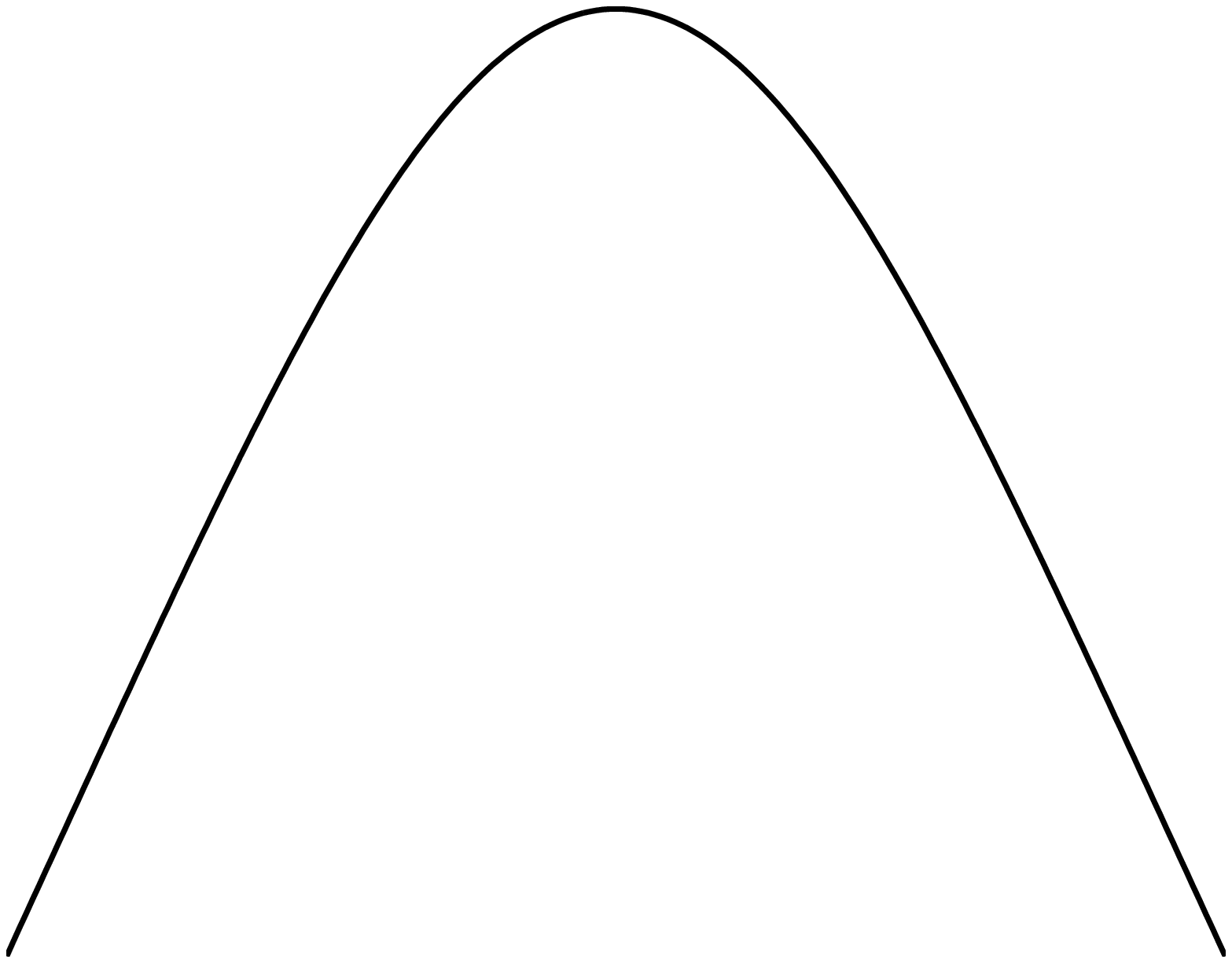}\\
	\\
	\\
	\end{array} & &\begin{array}{c} \includegraphics[width=4cm, height=3cm]{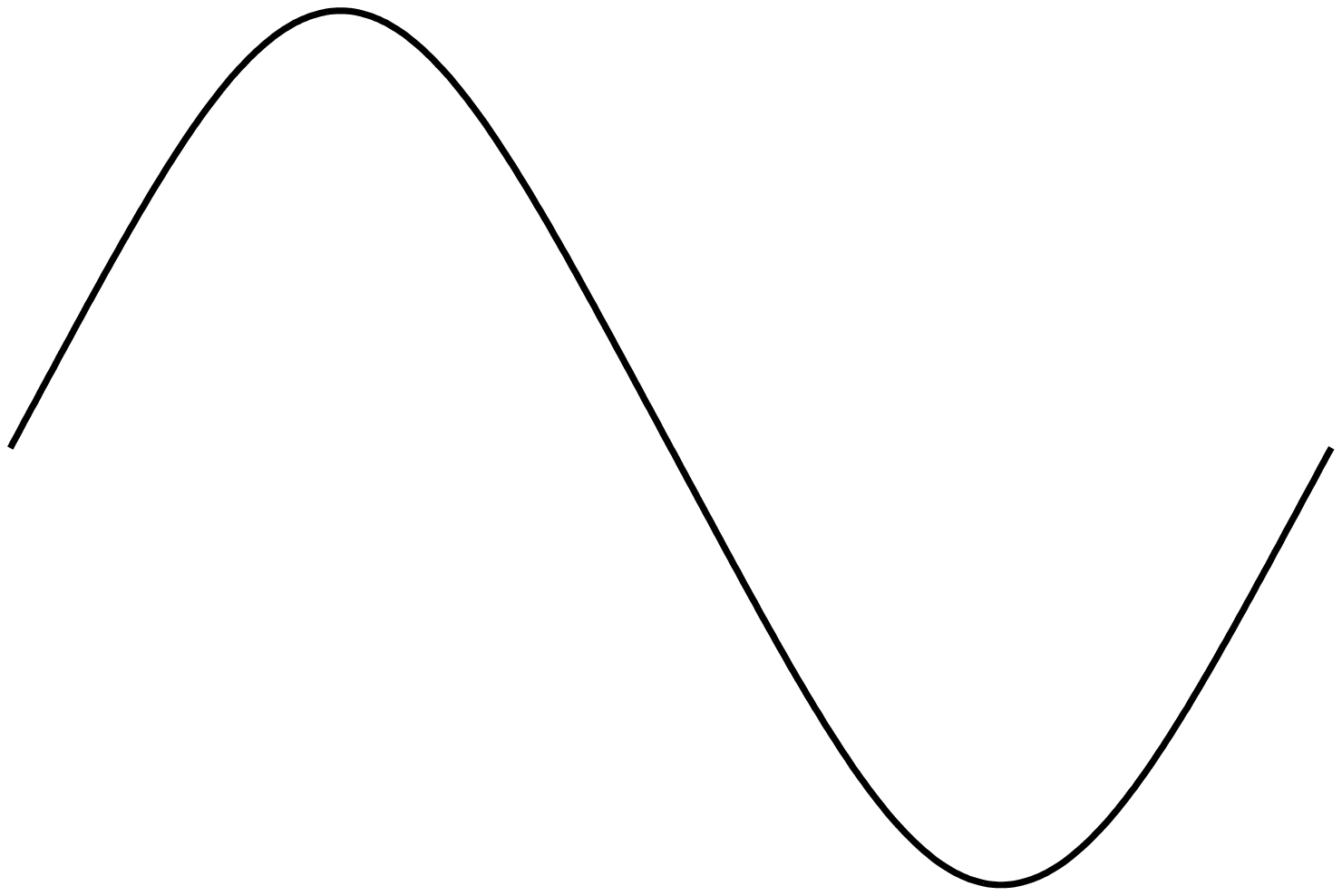}\\
	\end{array} &  &\begin{array}{c}	\includegraphics[width=4cm, height=4cm, angle=180, origin=c]{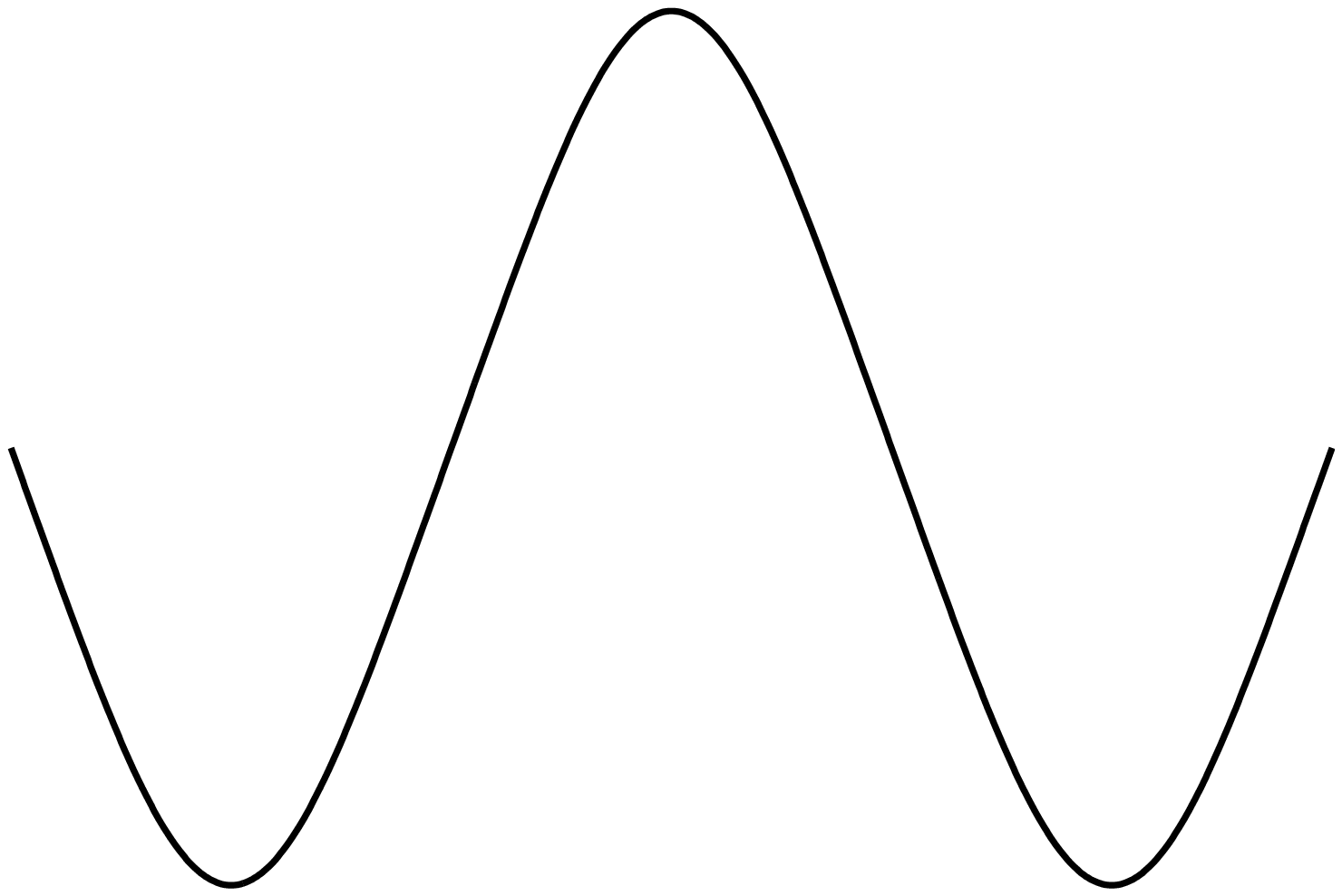} \\ \end{array}	\end{array}$$
	\caption{Examples of unstable equilibria: $z$, $z^{(2)}$ and $z^{(-3)}$.}
	\label{fig0}
\end{figure}

\subsection{Definition of solution for the SPDE}\label{secsol}

In general, equations like \eqref{MainSPDE} do not admit strong solutions in the
usual sense as they may not be globally defined but instead defined \textit{up
to an explosion time}. In the following we formalize the idea of explosion and
properly define the concept of solutions of \eqref{MainSPDE}.

First, we fix a probability space $(\Omega,\F,P)$ on which we have defined a
Brownian sheet
$$
W=\{ W{(t,x)} : (t,x) \in \R^+ \times [0,1]\},
$$ i.e. a stochastic process satisfying the following properties:
\begin{enumerate}
\item [i.] $W$ has continuous paths, i.e. $(t,x) \mapsto W{(t,x)}(\omega)$ is
continuous for every $\omega \in \Omega$.
\item [ii.] $W$ is a centered Gaussian process with covariance structure given
by
$$
\text{Cov}( W{(t,x)} , W{(s,y)} ) = (t \wedge s)(x \wedge y)
$$ for every $(t,x),(s,y) \in \R^+ \times [0,1]$.
\end{enumerate}
Then, for every $t \geq 0$ we define
$$
\G_t = \sigma( W{(s,x)} : 0 \leq s \leq t , x \in [0,1])
$$ and denote its augmentation by $\F_t$.\footnote{This means that $\F_t =
\sigma( \G_t \cup \mathcal{N})$ where $\mathcal{N}$ denotes the class of all
$P$-null sets of $\G_\infty = \sigma( \G_t : t \in \R^+)$.} The family
$(\F_t)_{t \geq 0}$ constitutes a filtration on $(\Omega,\F)$.
A \textit{solution up to an explosion time} of the equation \eqref{MainSPDE} on
$(\Omega,\F,P)$ with respect to the Brownian sheet $W$ and with initial datum $u
\in C_D([0,1])$
is a stochastic process $U^{u,\varepsilon} = \{ U^{u,\varepsilon}{(t,x)} : (t,x)
\in \R^+ \times [0,1]\}$ satisfying the following properties:

\noindent \begin{itemize}
\item [i.] $U^{u,\varepsilon}(0,\cdot)\equiv u$
\item [ii.] $U^{u,\varepsilon}$ has continuous paths taking values in
$\overline{\R}:=\R \cup \{\pm \infty\}$.
\item [iii.] $U^{u,\varepsilon}$ is adapted to the filtration $(\F_t)_{t \geq
0}$, i.e. for every $t \geq 0$ the mapping
$$
(\omega,x) \mapsto U^{u,\varepsilon}{(t,x)}(\omega)
$$ is $\F_t \otimes \B([0,1])$-measurable.
\item [iv.] If $\Phi$ denotes the fundamental solution of the heat equation on
the interval $[0,1]$ with homogeneous Dirichlet boundary conditions, which is
given by the formula
$$
\Phi(t,x,y) = \frac{1}{\sqrt{4\pi t}} \sum_{n \in \Z} \left[ \exp\left( -
\frac{(2n+y -x)^2}{4t} \right) - \exp\left( - \frac{(2n+y +x)^2}{4t}
\right)\right],
$$ and for $n \in \N$ we define the stopping time $\tau^{(n),u}_\varepsilon :=
\inf\{t>0 : \|U^{u,\varepsilon}{(t,\cdot)}\|_{\infty} \geq n\}$ then for every
$n \in \N$ we have $P$-a.s.: \\
\begin{enumerate}
\item [$\bullet$]
$\displaystyle{\int_0^1\int_{0}^{t\wedge{\tau^{(n),u}_\varepsilon}}
|\Phi(t\wedge \tau^{(n),u}_\varepsilon -s,x,y)g(U^{u,\varepsilon}(s,y))|dsdy  <
+\infty \text{  for all } t \in \R^+}$\\ \, \\
\item [$\bullet$] $\displaystyle{U^{u,\varepsilon}(t \wedge
\tau^{(n),u}_\varepsilon,x) = I_H^{(n)}(t,x)+ I_N^{(n)}(t,x) \text{  for
all }(t,x) \in \R^+ \times [0,1]},$ where
$$
I_H^{(n)}(t,x) = \int_0^1 \Phi(t \wedge \tau^{(n),u}_\varepsilon,x,y)u(y)dy
$$ \begin{center}
and
\end{center}
$$
I_N^{(n)}(t,x) =  \int_0^{t \wedge \tau^{(n),u}_\varepsilon} \int_0^1
\Phi(t\wedge \tau^{(n),u}_\varepsilon -s,x,y) \left(
g(U^{u,\varepsilon}(s,y))dyds + \varepsilon dW(s,y)\right).
$$
\end{enumerate}
\end{itemize}
We call the random variable $\tau^u_\varepsilon:=\lim_{n \rightarrow +\infty}
\tau^{(n),u}_\varepsilon$ the {\em explosion time} of $U^{u,\varepsilon}$.
Notice that the assumption of \mbox{continuity of $U^{u,\varepsilon}$} over
$\overline{\R}$ implies that:
\begin{enumerate}
\item [$\bullet$] $\tau^u_\varepsilon = \inf \{ t > 0 :
\|U^{u,\varepsilon}{(t,\cdot)}\|_\infty =+\infty\}$
\item [$\bullet$] $\|U^{u,\varepsilon}{((\tau^u_\varepsilon)^-,\cdot)}\|_\infty
= \|U^{u,\varepsilon}{(\tau^u_\varepsilon,\cdot)}\|_\infty =+\infty\,\,\mbox{ on
}\,\,\{ \tau^u_\varepsilon < +\infty\}.$
\end{enumerate}
We stipulate that $U^{u,\varepsilon}{(t,\cdot)}\equiv
U^{u,\varepsilon}(\tau^u_\varepsilon,\cdot)$ for $t \geq \tau$ whenever
$\tau^u_\varepsilon < +\infty$ but we do not assume that $\lim_{t \to
+\infty} U^{u,\varepsilon}{(t,\cdot)}$ exists if $\tau^u_\varepsilon=+\infty$.
Furthermore, since any initial datum $u \in C_D([0,1])$ is bounded, we always
have $P( \tau^u_\varepsilon > 0) = 1$. It can be shown that there exists a
(pathwise) unique solution $U^{u,\varepsilon}$ of \eqref{MainSPDE} up to an
explosion time and that it has the strong Markov property, i.e. if $\tilde
\tau$ is a stopping time of $U^{u,\varepsilon}$ then, conditional on
$\tilde\tau<\tau^u_\varepsilon$ and $U^{u,\varepsilon}{(\tilde\tau,\cdot)}=v$,
the future $\{ U^{u,\varepsilon}{(t + \tilde \tau,\cdot)} \colon 0 <
t<\tau^u_\varepsilon-\tilde\tau\}$ is independent of the past $\{
U^{u,\varepsilon}{(s,\cdot)} \colon 0 \leq s\le \tilde\tau \}$ and identical in
law to the solution of \eqref{MainSPDE} with initial datum $v$. We refer to
\cite{IM,W} for details.

\subsection{Local Freidlin-Wentzell estimates}\label{secLDP}

One of the main tools we use in the study of solutions of \eqref{MainSPDE} is
the local large deviations principle we briefly describe next.

Given $u \in C_D([0,1])$ and $T > 0$, we consider the metric space of continuous
functions
$$
C_{D_u}([0,T] \times [0,1]) = \{ v \in C([0,T]\times[0,1]) : v(0,\cdot)=u \text{
and }v(\cdot,0)=v(\cdot,1)\equiv 0 \}
$$ with the distance $d_T$ induced by the supremum norm, i.e. for $v,w \in
C_{D_u}([0,T]\times[0,1])$
$$
d_T(v,w) := \sup_{(t,x) \in [0,T]\times [0,1]} | v(t,x) - w(t,x) |,
$$ and define the rate function $I^u_T : C_{D_u}([0,T]\times [0,1]) \rightarrow
[0,+\infty]$ by the formula
$$
I^u_T (\varphi) = \left\{ \begin{array}{ll} \displaystyle{\frac{1}{2} \int_0^T
\int_0^1 |\p_t \varphi - \p_{xx} \varphi - g(\varphi)|^2} & \text{ if }\varphi
\in W^{1,2}_2([0,T]\times[0,1]) \,,\,\varphi(0,\cdot) = u \\ \\ +\infty &
\text{otherwise.}\end{array}\right.
$$
Here $W^{1,2}_2([0,T]\times[0,1])$ is the closure of $C^\infty([0,T] \times
[0,1])$ with respect to the norm
$$
\| \varphi \|_{W^{1,2}_2} = \left( \int_0^T \int_0^1 \left[ |\varphi|^2 + |\p_t
\varphi|^2 + |\p_x \varphi|^2 + |\p_{xx} \varphi|^2\right]\right)^\frac{1}{2},
$$ i.e. the Sobolev space of square-integrable functions defined on $[0,T]\times
[0,1]$ with one square-integrable weak \mbox{time derivative} and two
square-integrable weak space derivatives.

By following the lines of \cite{B1,FJL,SOW}, it is possible to establish a
large deviations principle for solutions of \eqref{MainSPDE} with rate function
$I$ as given above whenever the source term $g$ is globally Lipschitz.
Unfortunately, this is not the case for us. Nonetheless, by employing
localization arguments like the ones carried out in \cite{GS}, one can obtain a
weaker version of this principle which only holds locally, i.e. while the
process remains inside any fixed bounded region. More precisely, we have the
following result.

\begin{theorem} If for each $n \in \N$ and $u \in C_D([0,1])$ we define
$$
\tau^{(n),u}:= \inf\{ t > 0 : \| U^u(t,\cdot)\|_\infty \geq
n\}\hspace{1cm}\text{ and }\hspace{1cm}\mathcal{T}^{(n),u}_\varepsilon:=
\tau^{(n),u}_\varepsilon \wedge \tau^{(n),u}
$$ where $\tau^{(n),u}_\varepsilon$ is defined as in Section \ref{secsol}, then
the following estimates hold:
\begin{enumerate}
\item [$\bullet$] \textnormal{\bf Lower bound}. For any $\delta,h > 0$ and $n \in \N$, there
exists $\varepsilon_0$ such that
\begin{equation}\label{LDP1}
P\left( d_{T \wedge \mathcal{T}^{(n),u}_\varepsilon} \left( U^{u,\varepsilon},
\varphi\right) < \delta \right) \geq e^{-\frac{I_T^u(\varphi)+h}{\varepsilon^2}}
\end{equation} for all $0 < \varepsilon < \varepsilon_0$, $u \in C_D([0,1])$ and
$\varphi \in C_{D_u}([0,T]\times [0,1])$ with $\| \varphi \|_\infty \leq n$. \\
\item [$\bullet$] \textnormal{\bf Upper bound}. For any $\delta > 0$ there exist
$\varepsilon_0 > 0$ and $C > 0$ such that
\begin{equation}
\label{grandes1}
\sup_{u \in C_D([0,1])} P\left( d_{ T \wedge
\mathcal{T}^{(n),u}_\varepsilon}\left(U^{u,\varepsilon}, U^u\right) > \delta
\right) \leq e^{-\frac{C}{\varepsilon^2}},
\end{equation}for all $0 < \varepsilon < \varepsilon_0$.
\end{enumerate}
\end{theorem}

The usual large deviations estimates for these type of systems usually feature a
more refined version of the upper bound than the one we give here (see
\cite{B1}, for example). However, the estimate in \eqref{grandes1} is enough for
our purposes and so we do not pursue any generalizations of it here. Also, notice that both estimates
are somewhat uniform in the initial datum. This uniformity is obtained as in
\cite{B1} by using the fact that $g$ is Lipschitz when restricted to bounded
sets. We refer to \cite{B1, FJL} for further details.

\subsection{Main results}\label{mainresults}

Our purpose is to study the asymptotic behavior as $\varepsilon \rightarrow 0$
of $U^{u,\varepsilon}$, the solution of \eqref{MainSPDE}, for the different
initial data $u \in C_{D}([0,1])$. We now state our results. In the
following we write $P_u$ to denote the distribution of $U^{u,\varepsilon}$. Whenever we
choose to make the initial datum clear in this way, we will drop the superscript $u$ from the notation for simplicity purposes.

Our first result is concerned with the continuity of the explosion time for
initial data in the domain of explosion $\mathcal{D}_e$. In this case one
expects the stochastic and deterministic systems to both exhibit a similar
behavior for $\varepsilon > 0$ sufficiently small, since then the noise will not
be able to grow fast enough so as to overpower the quickly exploding source
term.
We show this to be truly the case for $u \in \mathcal{D}_e$ such that
\mbox{$U^u$ remains bounded from one side.}

\begin{theorem}\label{thm1} Let $\mathcal{D}^*_e$ be the set of
initial data $u \in \mathcal{D}_e$ such that $U^u$ explodes only through one
side, i.e. $U^{u}$ remains bounded either from below or above until its
explosion time $\tau^u$. Then given $\delta > 0$ and a bounded set $\mathcal{K}
\subseteq \mathcal{D}_e^*$ at a positive distance from $\p \mathcal{D}^*_e$
there exists a constant $C > 0$ such that
$$
\sup_{u \in \mathcal{K}} P_u ( |\tau_\varepsilon - \tau| > \delta ) \leq
e^{-\frac{C}{\varepsilon^2}}.
$$
\end{theorem}

The main differences in behavior between the stochastic and deterministic systems \mbox{appear for} initial data in $\mathcal{D}_\0$, where
metastable behavior is observed. According to the characterization of
metastability for stochastic processes in \cite{CGOV, GOV}, this behavior is given by two facts: the time averages of the process
remain stable until an abrupt transition occurs and a different value is
attained; furthermore, the time of this transition is unpredictable in the sense
that, when suitably rescaled, it should have an exponential distribution. \mbox{We
manage to establish} this description rigorously for our system whenever $1 < p <
5$. This rigorous description is contained in the remaining results.

Define the quantity $\Delta := 2(S(z) - S(\mathbf{0}))$. Our next result states that for any $u \in \mathcal{D}_{\mathbf{0}}$ the asymptotic magnitude of $\tau^u_\varepsilon$ is, up to logarithmic equivalence, of order $e^{\frac{\Delta}{\varepsilon^2}}$.

\begin{theorem}\label{thm2} Given $\delta > 0$ and
a bounded set $\mathcal{K} \subseteq \mathcal{D}_{\mathbf{0}}$ at a positive
distance from $\p \mathcal{D}_{\mathbf{0}}$, if $1 < p < 5$ then we have
$$
\lim_{\varepsilon \rightarrow 0} \left[ \sup_{u \in \mathcal{K}} \left| P_u
\left( e^{\frac{\Delta - \delta}{\varepsilon^2}} < \tau_\varepsilon <
e^{\frac{\Delta + \delta}{\varepsilon^2}}\right)-1\right|\right]=0.
$$
\end{theorem}

Theorem \ref{thm2} suggests that, for initial data $u \in
\mathcal{D}_{\mathbf{0}}$, the typical route of $U^{u,\varepsilon}$ towards
infinity involves passing through one of the unstable equilibria of minimal energy,
$\pm z$. This seems reasonable since, as we will see in Section \ref{secg}, for
$1 < p < 5$ the barrier imposed by the potential $S$ is the lowest there. The
following result establishes this fact rigorously.

\begin{theorem}\label{thm3} Given $\delta > 0$ and a bounded set
$\mathcal{K} \subseteq \mathcal{D}_{\mathbf{0}}$ at a positive distance from $\p
\mathcal{D}_{\mathbf{0}}$, if $1 < p < 5$ then we have
$$
\lim_{\varepsilon \rightarrow 0} \left[ \sup_{u \in \mathcal{K}} \left| P_u
\left( \tau_\varepsilon( \mathcal{D}_{\mathbf{0}}^c ) < \tau_\varepsilon\,,\,
U^{\varepsilon}( \tau_\varepsilon( \mathcal{D}_{\mathbf{0}}^c), \cdot) \in
B_\delta (\pm z) \right)-1\right|\right]=0,
$$ where $\tau^u_\varepsilon( \mathcal{D}_{\mathbf{0}}^c):= \inf\{ t > 0 :
U^{u,\varepsilon}(t,\cdot) \notin \mathcal{D}_{\mathbf{0}} \}$ and $B_\delta(\pm
z):= B_\delta (z) \cup B_\delta(-z)$.
\end{theorem}

Our next result is concerned with the the asymptotic loss of memory of
$\tau^u_\varepsilon$. For $\varepsilon > 0$ define the scaling coefficient
\begin{equation}\label{defibeta}
\beta_{\varepsilon}= \inf \{ t \geq 0 : P_{\mathbf{0}} ( \tau_\ve > t ) \leq
e^{-1} \}
\end{equation}
Observe that Theorem \ref{thm2} implies that the family $(\beta_\varepsilon)_{\varepsilon >
0}$
satisfies $\lim_{\varepsilon \rightarrow 0} \varepsilon
^{2}\log\beta_{\varepsilon} = \Delta$. This next result states that for any $u \in \mathcal{D}_{\mathbf{0}}$ the normalized explosion time
$\frac{\tau^u_\varepsilon}{\beta_\varepsilon}$ converges in distribution to an exponential random variable of mean one.

\begin{theorem}
 \label{thm4}
Given $\delta > 0$ and a bounded set $\mathcal{K}
\subseteq \mathcal{D}_{\mathbf{0}}$ at a positive distance from $\p
\mathcal{D}_{\mathbf{0}}$, if $1 < p < 5$ then for any $t > 0$ we have
$$
\lim_{\varepsilon \rightarrow 0} \left[ \sup_{u \in \mathcal{K}} \left| P_{u}
(\tau_{\varepsilon} > t\beta_{\varepsilon}) - e^{-t} \right| \right] = 0.
$$
\end{theorem}

Finally, we show the stability of time averages of continuous functions
evaluated along paths of the process starting in $\mathcal{D}_{\mathbf{0}}$,
i.e. they remain close to the value of the \mbox{function at $\mathbf{0}$.}
These time averages are taken along intervals of length going to infinity and
times may be taken as being almost (in a suitable scale) the explosion time.
This tells us that, up until the explosion time, the system spends most of its
time in a small \mbox{neighborhood of $\mathbf{0}$.}

\begin{theorem}\label{thm5} There exists a sequence
$(R_\varepsilon)_{\varepsilon > 0}$ with $\lim_{\varepsilon \rightarrow 0}
R_\varepsilon = +\infty$ and $\lim_{\varepsilon \rightarrow 0}
\frac{R_\varepsilon}{\beta_\varepsilon} = 0$ such that given $\delta > 0$ for
any bounded set $\mathcal{K} \subseteq \mathcal{D}_{\mathbf{0}}$ at a positive
\mbox{distance from $\mathcal{W}$} we have
$$
\lim_{\varepsilon \rightarrow 0} \left[ \sup_{u \in \mathcal{K}} P_u \left(
\sup_{0 \leq t \leq \tau_\varepsilon - 3R_\varepsilon}\left|
\frac{1}{R_\varepsilon}\int_t^{t+R_\varepsilon} f(U^{\varepsilon}(s,\cdot))ds -
f(\mathbf{0})\right| > \delta \right) \right] = 0
$$ for any bounded continuous function $f: C_D([0,1]) \rightarrow \R$.
\end{theorem}

Theorem \ref{thm1} is proved in Section \ref{explo}, the remaining results are
proved in \mbox{Sections \ref{secescapedeg} and \ref{secfinal}.}
Perhaps the proof of Theorem \ref{thm1} is where one finds the most differences
with other works in the literature dealing with similar problems, namely
\cite{GOV,MOS}. This is due to the \mbox{fact that} for this part we cannot use large
deviations estimates as these do. The remaining results were established in
\cite{B1,MOS} for the tunneling time in an infinite-dimensional double-well
potential model, i.e. the time the system takes to go from one well to the
bottom of the other one. Our proofs are similar to the ones found in these
references, although we have the additional difficulty of dealing with solutions
which are not globally defined.

\section{Phase diagram of the deterministic system}\label{deter}

In this section we review the behavior of solutions to \eqref{MainPDE} for the
different initial data in $C_D([0,1])$. We begin by characterizing the unstable
equilibria of the system.

\begin{prop}\label{equilibrios} A function $w \in C_D([0,1]) $ is an equilibrium
of the system \mbox{if and only if} there exists $n \in \Z$ such that $w =
z^{(n)}$, where for each $n \in \N$ we define $z^{(n)} \in C_D([0,1])$ by the
formula
$$
z^{(n)}(x)= \left\{ \begin{array}{rl} n^{\frac{2}{p-1}}z( nx - [nx] ) & \text{
if $[nx]$ is even} \\ \\ - n^{\frac{2}{p-1}}z( nx - [nx] )& \text{ if $[nx]$ is
odd}\end{array}\right.
$$ and also define $z^{(-n)}:= - z^{(n)}$ and $z^{(0)}:= \mathbf{0}$.
Furthermore, for each $n \in \Z$ we have
$$
\| z^{(n)} \|_\infty = |n|^{\frac{2}{p-1}}\|z\|_\infty \hspace{2cm}\text{ and
}\hspace{2cm} S(z^{(n)}) = |n|^{2 \left(\frac{p+1}{p-1}\right)} S(z).
$$
\end{prop}

\begin{proof} It is simple to verify that for each $n \in \Z$ the function
$z^{(n)}$ is an equilibrium of the system and that each $z^{(n)}$ satisfies both
$\| z^{(n)} \|_\infty = |n|^{\frac{2}{p-1}}\|z\|_\infty$ and $S(z^{(n)}) =
|n|^{2 \left(\frac{p+1}{p-1}\right)} S(z)$. Therefore, we must only check that
for any equilibrium of the system $w \in C_D([0,1]) - \{ \mathbf{0}\}$ there
exists $n \in \N$ such that $w$ coincides with either $z^{(n)}$ or $-z^{(n)}$.

Thus, let $w \in C_D([0,1]) - \{ \mathbf{0}\}$ be an equilibrium of \eqref{MainPDE}
and define the sets
$$
G^+ = \{ x \in (0,1) : w(x) > 0 \} \hspace{2cm}\text{ and }\hspace{2cm}G^- = \{x
\in (0,1) : w(x) < 0 \}.
$$ Since $w \neq \mathbf{0}$ at least one of these sets must be nonempty. On the
other hand, if only one of them is nonempty then, since $z$ is the unique
nonnegative equilibrium different from $\mathbf{0}$, we must have either $w=z$
or $w=-z$. Therefore, we may assume that both $G^+$ and $G^-$ are nonempty.
Notice that since $G^+$ and $G^-$ are open sets we may write them as
$$
G^+ = \bigcup_{k \in \N} I^{+}_k \hspace{2cm}\text{ and }\hspace{2cm}G^- =
\bigcup_{k \in \N} I^-_k
$$ where the unions are disjoint and each $I^{\pm}_k$ is a (possibly empty) open
interval.

We first show that each union must be finite. Take $k \in \N$ and suppose we can write
$I^+_k = (a_k, b_k)$ for some $0\leq a_k < b_k \leq 1$. It is easy to check that
$\tilde{w}_k : [0,1] \rightarrow \R$ given by
$$
\tilde{w}_k (x) := (b_k - a_k)^{\frac{2}{p-1}} w( a_k + (b_k-a_k) x)
$$ is a nonnegative equilibrium of the system different from $\mathbf{0}$ and
thus it must be $\tilde{w}_k = z$. This implies that $\| w \|_\infty \geq (b_k -
a_k)^{- \frac{2}{p-1}} \| \tilde{w}_k \|_\infty = (b_k - a_k)^{- \frac{2}{p-1}}
\| z \|_\infty$ from where we see that an infinite number of nonempty $I^+_k$
would contradict the fact that $w$ is bounded.
Thus, we see that $G^+$ is a finite union of open intervals and by symmetry so is $G^-$.
The same argument also implies that for each interval $I^\pm_k=(a_k, b_k)$ the
graph of $w|_{I^\pm_k}$ coincides with that of $\pm z$ but when scaled by the
factor $(b_k - a_k)^{-\frac{2}{p-1}}$. More precisely, for all $x \in [0,1]$ we
have
\begin{equation}\label{wkz2}
 w( a_k + (b_k-a_k) x)= \pm (b_k - a_k)^{-\frac{2}{p-1}}z(x).
\end{equation}

Now, Hopf's Lemma \cite[p.~330]{E} implies that $\partial_x z(0^+) >
0$ and $\partial_x z(1^-) < 0$. Furthermore, since $z$ is symmetric with respect
to $x=\frac{1}{2}$ we have in fact that $\partial_x z(0^+) = -
\partial_x z(1^-) > 0$. In light of \eqref{wkz2} and the fact that $w$ is everywhere
differentiable, the former tells us that plus and minus intervals must present
themselves in alternating order, that their closures cover all of $[0,1]$ and also that their lengths are all the same. Combining this with \eqref{wkz2}
we conclude the proof.
\end{proof}

As a consequence of Proposition \ref{equilibrios} we obtain the following
important corollary.

\begin{cor}\label{corzminimo} The functions $\pm z$ minimize the pontential $S$ and the infinity norm among all the unstable equilibria of \eqref{MainPDE}.
In particular, we have that $\inf_{u \in \mathcal{W}} S(u) = S(\pm z)$.
\end{cor}

\begin{proof} The first statement is clear from Proposition \ref{equilibrios}
while the second one is deduced from the first since the mapping $t \mapsto S(U^u(t,\cdot))$ is monotone decreasing and continuous for any $u \in H_0^1((0,1))$ (see Proposition \ref{Lyapunov}).
\end{proof}
Concerning the asymptotic behavior of solutions to \eqref{MainPDE}, the following
dichotomy \mbox{was proved} by Cortázar and Elgueta in \cite{CE1}.
\begin{prop}\label{descomp1} Let $U^u$ denote the solution to
\eqref{MainPDE} with initial datum $u \in C_D([0,1])$. Then one of these two
possibilities must hold:
\begin{enumerate}
\item [i.] $\tau^{u} < +\infty$ and $U^u$ blows up as $t \nearrow \tau^{u}$,
i.e. $\lim_{t \nearrow \tau^{u}} \|U^u(t,\cdot)\|_\infty = +\infty$
\item [ii.] $\tau^{u} = +\infty$ and $U^u$ converges to some stationary solution
$z^{(n)}$ as $t \rightarrow +\infty$.
\end{enumerate}
\end{prop}
Proposition \ref{descomp1} allows us to split the space $C_D([0,1])$ of initial
data into three parts
\begin{equation}\label{decomp12}
C_D([0,1]) = \mathcal{D}_{\mathbf{0}} \cup \mathcal{W} \cup \mathcal{D}_e
\end{equation} where $\mathcal{D}_{\mathbf{0}}$ denotes the domain of attraction of
the origin $\mathbf{0}$, $\mathcal{W}$ is the union of all stable manifolds of
the unstable equilibria and $\mathcal{D}_e$ is the domain of explosion of the system, i.e. the
set of all initial data for which the system explodes in finite time. It can be
seen that both $\mathcal{D}_{\mathbf{0}}$ and $\mathcal{D}_e$ are open sets and
that $\mathcal{W}$ is the common boundary separating them. The following
proposition gives a useful characterization of the domain of explosion
$\mathcal{D}_e$.

\begin{prop}[{\cite[Theorem~17.6]{QS}}]\label{caract} The domain of explosion $\mathcal{D}_e$ satisfies
$$
\mathcal{D}_e = \{ u \in C_D([0,1]) : S( U^u (t, \cdot) ) < 0 \text{ for some }0
\leq t < \tau^u \}.
$$ Furthermore, we have $\lim_{t \nearrow \tau^u} S( U^u(t,\cdot) ) =
-\infty$.
\end{prop}

From these results one can obtain a precise description of the
domains $\mathcal{D}_{\mathbf{0}}$ and $\mathcal{D}_e$ in the region of
nonnegative data. Cortázar and Elgueta proved the following result in \cite{CE2}.

\begin{prop}\label{descomp2} $\,$
\begin{enumerate}
\item [i.] Assume $u \in C_D([0,1])$ is nonnegative and such that $U^u$ is
globally defined and converges to $z$ as $t \rightarrow +\infty$. Then for $v
\in C_D([0,1])$ we have that:
\begin{enumerate}
\item [$\bullet$] $\mathbf{0} \lneq v \lneq u \Longrightarrow U^v$ is globally
defined and converges to $\mathbf{0}$ as $t \rightarrow +\infty$.
\item [$\bullet$] $u \lneq v \Longrightarrow U^v$ explodes in finite time.
\end{enumerate}
\item [ii.] For every nonnegative $u \in C_D([0,1])$ there exists $\lambda_c^u >
0$ such that for every $\lambda > 0$:
\begin{enumerate}
\item [$\bullet$] $0 < \lambda < \lambda_c^u \Longrightarrow U^{\lambda u}$ is
globally defined and converges to $\mathbf{0}$ as $t \rightarrow +\infty$.
\item [$\bullet$] $\lambda = \lambda_c^u \Longrightarrow U^{\lambda u}$ is
globally defined and converges to $z$ as $t \rightarrow +\infty$.
\item [$\bullet$] $\lambda > \lambda_c^u \Longrightarrow U^{\lambda u}$ explodes
in finite time.
\end{enumerate}
\end{enumerate}
\end{prop}

This last result yields the existence of an unstable manifold of the saddle
point $z$ which is contained in the region of nonnegative initial data and which we shall
denote by $\mathcal{W}^z_u$. It is $1$-dimensional, has nonempty
intersection with both $\mathcal{D}_{\mathbf{0}}$ and $\mathcal{D}_e$ and joins
$z$ with $\mathbf{0}$. By symmetry, a similar description also holds for the
opposite unstable equilibrium $-z$. Figure \ref{fig1} depicts the decomposition
in \eqref{decomp12} together with the unstable manifolds $\mathcal{W}^{\pm
z}_u$. \mbox{By exploiting the} structure of the remaining unstable equilibria
given by Proposition \ref{equilibrios} one can verify for each of them the
analogue of (ii) in Proposition \ref{descomp2}. Details in \cite{SJS}.

\psfrag{U1}{\vspace{-155pt}$U\equiv 0$}
\psfrag{U2}{\vspace{-45pt}$U \equiv \1$}
\psfrag{De}{\hspace{60pt}\vspace{30pt}$\mathcal{D}_e$}
\psfrag{D0}{$\mathcal{D}_0$}
\psfrag{Wu}{$\mathcal W_1^u$}
\psfrag{Ws}{$\mathcal W_1^s$}
\begin{figure}
	\centering
	\includegraphics[width=8cm]{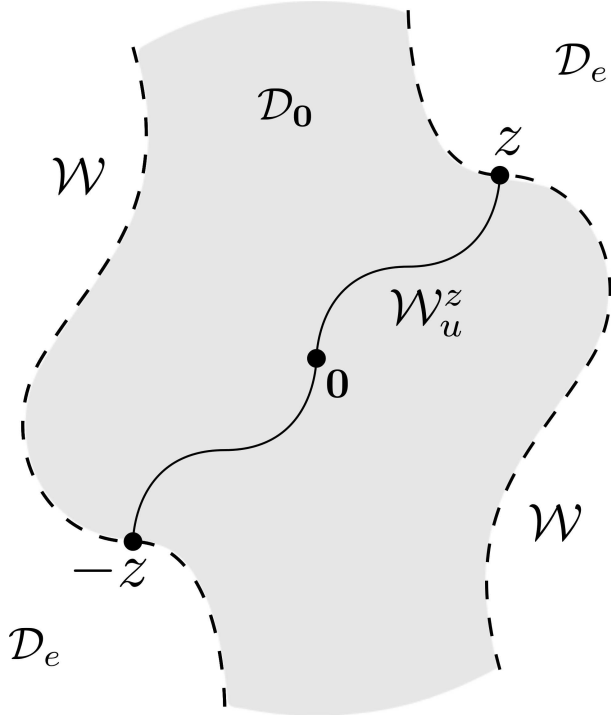}
	\caption{The phase diagram of equation \eqref{MainPDE}.}
	\label{fig1}
\end{figure}

\section{Asymptotic behavior of $\tau_\varepsilon^u$ for $u \in
\mathcal{D}_e$}\label{explo}

In this section we investigate the continuity properties of the explosion time
$\tau_\varepsilon^u$ for initial data in the domain of explosion
$\mathcal{D}_e$. We show that, under suitable conditions on the initial datum $u
\in \mathcal{D}_e$, the random explosion time $\tau_\varepsilon^u$ converges in probability to the deterministic explosion time $\tau^u$ as $\ve \to 0$.
To be more precise, let us consider the sets of initial data in
$\mathcal{D}_e$ which explode only through $+\infty$ or $-\infty$, i.e.
$$
\mathcal{D}_e^+ = \left\{ u \in \mathcal{D}_e : \inf_{(t,x) \in [0,\tau^u)
\times [0,1]}  U^u(t,x) > -\infty \right\}
$$
and
$$
\mathcal{D}_e^- = \left\{ u \in \mathcal{D}_e : \sup_{(t,x) \in [0,\tau^u)
\times [0,1]} U^u(t,x) < +\infty \right\}.
$$ Notice that $\mathcal{D}_e^+$ and $\mathcal{D}_e^-$ are disjoint and also
that they satisfy the relation $\mathcal{D}_e^- = - \mathcal{D}_e^+$.
Furthermore, we shall see below that $\mathcal{D}_e^+$ is an open set. Let us
write $\mathcal{D}_e^*:= \mathcal{D}_e^+ \cup \mathcal{D}_e^-$. The result we
are to prove is the following.

\begin{theorem}\label{contexp} For any bounded set $\mathcal{K} \subseteq
\mathcal{D}_e^*$ at a positive distance from $\p \mathcal{D}_e^*$ and $\delta >
0$ there exists a constant $C > 0$ such that
$$
\sup_{u \in \mathcal{K}} P_u ( |\tau_\varepsilon - \tau| > \delta ) \leq e^{-
\frac{C}{\varepsilon^2}}.
$$
\end{theorem} We split the proof of Theorem \ref{contexp} into two parts:
proving first a lower bound and then an upper bound for $\tau_\varepsilon$. The
first one is a consequence of the continuity of solutions to \eqref{MainSPDE}
with respect to $\ve$ on intervals where the deterministic solution
\mbox{remains bounded.} The precise estimate is contained in the following
proposition.

\begin{prop}\label{convergenciainferior0} For any bounded set $\mathcal{K}
\subseteq \mathcal{D}_e$ and $\delta > 0$ there exists a constant $C > 0$ such
that
\begin{equation}\label{convergenciainferior}
\sup_{u \in \mathcal{K}} P_u ( \tau_\varepsilon < \tau - \delta ) \leq e^{-
\frac{C}{\varepsilon^2}}.
\end{equation}
\end{prop}

\begin{proof} From the continuity of solutions respect to the initial datum (see Proposition \ref{G.2}) we have that $\inf_{u
\in \mathcal{K}} \tau^u > 0$ so that we may assume $\tau^u > \delta$ for all $u \in \mathcal{K}$. For each $u \in \mathcal{D}_e$ we define the quantity
$$
M_u := \sup_{0 \leq t \leq \max\{0, \tau^u - \delta \}}
\|U^{u}(t,\cdot)\|_\infty.
$$ By the continuity of solutions once again we obtain that the application $u \mapsto M_u$ is both upper semicontinuous and finite on
$\mathcal{D}_e$ and hence, by Propositions \ref{G.1} and \ref{A.2},
we conclude that $M:= \sup_{u \in \mathcal{K}} M_u < +\infty$. Similarly, since
$u \mapsto \tau^u$ is both continuous and finite on $\mathcal{D}_e$ (see
Corollary \ref{contdetexp} below for a proof of this) we also obtain that
$\mathcal{T}:= \sup_{u \in \mathcal{K}} \tau^u < +\infty$. Hence, for $u \in
\mathcal{K}$ we get
$$
P_u (\tau^u_\varepsilon < \tau^u - \delta ) \leq P_u \left( d_{\tau^u -
\delta}\left(U^{M_u+1,\varepsilon},U^{M_u+1}\right) > \frac{1}{2}\right) \leq
P_u \left( d_{\mathcal{T} - \delta}\left(U^{M+1,\varepsilon},U^{M+1}\right) >
\frac{1}{2}\right).
$$ By the estimate \eqref{grandes1} we conclude \eqref{convergenciainferior}.
\end{proof}

To establish the upper bound we consider for each $u \in \mathcal{D}_e^+$ the
process
$$
Z^{u,\varepsilon} := U^{u,\varepsilon} - V^{\mathbf{0},\varepsilon}
$$ where $U^{u,\varepsilon}$ is the solution of \eqref{MainSPDE} with initial
datum $u$ and $V^{\mathbf{0},\varepsilon}$ is the solution of \eqref{MainSPDE}
with source term $g \equiv 0$ and initial datum $\mathbf{0}$ constructed using
the same Brownian sheet
as $U^{u,\varepsilon}$. Let us observe that
$Z^{u,\varepsilon}$ satisfies the random partial differential equation
\begin{equation}\label{randomPDE}
\left\{\begin{array}{rll}
\p_t Z^{u,\varepsilon} &= \p^2_{xx}Z^{u,\varepsilon} + g(Z^{u,\varepsilon} +
V^{\mathbf{0},\varepsilon}) & \quad t>0 \,,\, 0<x<1 \\
Z^{u,\varepsilon}(t,0)&=Z^{u,\varepsilon}(t,1)=0 & \quad t>0 \\
Z^{u,\varepsilon}(0,x) &=u(x).
\end{array}\right.
\end{equation}Furthermore, since $V^{\mathbf{0},\varepsilon}$ is globally
defined and remains bounded on finite time intervals, we have that
$Z^{u,\varepsilon}$ and $U^{u,\varepsilon}$ share the same explosion time.
Hence, to  obtain the desired upper bound on $\tau^u_\varepsilon$ we may study
the behavior of $Z^{u,\varepsilon}$. The advantage of this approach is that, in
general, the behavior of $Z^{u,\varepsilon}$ will be easier to understand than
that of $U^{u,\varepsilon}$. Indeed, each realization of $Z^{u,\varepsilon}$ is
the solution of a partial differential equation which one can handle with PDE techniques.

Now, a straightforward calculation using the mean value theorem shows that
whenever $\| V^{\mathbf{0},\varepsilon} \|_\infty < 1$ the process
$Z^{u,\varepsilon}$ satisfies the inequality
\begin{equation}\label{eqcomparacion}
\p_t Z^{u,\varepsilon} \geq \p^2_{xx} Z^{u,\varepsilon} + g(Z^{u,\varepsilon}) -
h|Z^{u,\varepsilon}|^{p-1} - h
\end{equation} where $h := p2^{p-1}\| V^{\mathbf{0},\varepsilon} \|_\infty > 0$.
Therefore, to establish the upper bound on $\tau_\varepsilon^u$ we first consider for
$h > 0$
the solution $\underline{Z}^{(h),u}$ to the equation
\begin{equation}\label{randomPDE2}
\left\{\begin{array}{rll}
\p_t \underline{Z}^{(h),u} &= \p^2_{xx}\underline{Z}^{(h),u} +
g(\underline{Z}^{(h),u}) - h|\underline{Z}^{(h),u}|^{p-1} - h  & \quad t>0 \,,\,
0<x<1 \\
\underline{Z}^{(h),u}(t,0)&=\underline{Z}^{(h),u}(t,1)=0 & \quad t>0 \\
\underline{Z}^{(h),u}(0,x) &=u(x).
\end{array}\right.
\end{equation} and obtain a convenient upper bound for the explosion time of
this new process valid for every $h$ sufficiently small. By showing then that for $h$ suitably small the process $\underline{Z}^{(h),u}$ explodes
through $+\infty$, the fact that $Z^{u,\varepsilon}$ is a supersolution to
\eqref{randomPDE2} will yield the desired upper bound on the explosion time of
$Z^{u,\varepsilon}$, provided that $\| V^{\mathbf{0},\varepsilon} \|_\infty$ remains small
enough. For this last part is where the assumption that $u \in \mathcal{D}_e^+$ is necessary.
Lemma \ref{expestimate} below contains the proper estimate on $\underline{\tau}^{(h),u}$,
the explosion time of $\underline{Z}^{(h),u}$.

\begin{defi} For $h \geq 0$ we define the potential $\underline{S}^{(h)}$
on $C_D([0,1])$ associated to \eqref{randomPDE2} by the formula
$$
\underline{S}^{(h)}(v)=\left\{ \begin{array}{ll} \displaystyle{\int_0^1
\left[\frac{1}{2} \left(\frac{dv}{dx}\right)^2 - \frac{|v|^{p+1}}{p+1} + hg(v) +
hv\right]} & \text{ if $v \in H^1_0((0,1))$ }\\ \\ +\infty & \text{
otherwise.}\end{array}\right.
$$ Notice that $\underline{S}^{(0)}$ coincides with our original potential $S$.
Moreover, it is easy to check that for all $h \geq 0$ the potential
$\underline{S}^{(h)}$ satisfies all properties established for $S$ in the
appendix.
\end{defi}

\begin{lema}\label{expestimate} Given $\delta > 0$ there exists $M > 0$ such
that:
\begin{enumerate}
\item [i.] For every $0 \leq h < 1$, any $u \in C_D([0,1])$ with $\underline{S}^{(h)}(u) \leq - \frac{M}{2}$ verifies $\underline{\tau}^{(h),u} <
\frac{\delta}{2}$.
\item [ii.] Given $K > 0$ there exist constants $\rho_{M,K}, h_{M,K} > 0$
depending only on $M$ and $K$ such that any $u \in C_D([0,1])$ satisfying $S(u)
\leq - M$ and $\| u\|_\infty \leq K$ verifies
$$
\sup_{v \in B_{\rho_{M,K}}(u)} \underline{\tau}^{(h),v} < \delta
$$ for all $0 \leq h < h_{M,K}$.
\end{enumerate}
\end{lema}

\begin{proof} Let us take $\delta > 0$ and show first that (i) holds for an
appropriate choice of $M$. \mbox{For fixed $M > 0$} and $0 \leq h < 1$, let $u
\in C_D([0,1])$ be such that $\underline{S}^{(h)}(u) \leq - \frac{M}{2}$ and
consider the application $\phi^{(h),u}: [0, \tau^{(h),u}) \rightarrow \R^+$
given by the formula
$$
\phi^{(h),u}(t) = \int_0^1 \left(\underline{Z}^{(h),u}(t,x)\right)^2dx.
$$ It is simple to verify that $\phi^{(h),u}$ is continuous and that for any
$t_0 \in (0,\underline{\tau}^{(h),u})$ it satisfies
\begin{equation}\label{eqpoten1}
\frac{d\phi^{(h),u}}{dt}(t_0) \geq - 4\underline{S}^{(h)}(u^{(h)}_{t_0}) + 2
\int_0^1 \left[ \left(\frac{p-1}{p+1}\right) |u^{(h)}_{t_0}|^{p+1} - h
\left(\frac{p+2}{p}\right)|u^{(h)}_{t_0}|^p - h|u^{(h)}_{t_0}|\right]
\end{equation} where we write $u^{(h)}_{t_0}:=\underline{Z}^{(h),u}(t_0,\cdot)$
for convenience. Hölder's inequality reduces \eqref{eqpoten1} to
\begin{equation}\label{eqpoten2}
\frac{d\phi^{(h),u}}{dt}(t_0) \geq - 4\underline{S}^{(h)}(u^{(h)}_{t_0}) +
2\left[ \left(\frac{p-1}{p+1}\right) \|u^{(h)}_{t_0}\|_{L^{p+1}}^{p+1} -
h\left(\frac{p+2}{p}\right)\|u^{(h)}_{t_0}\|_{L^{p+1}}^{p} -
h\|u^{(h)}_{t_0}\|_{L^{p+1}}\right].
\end{equation} Observe that, by definition of $\underline{S}^{(h)}$ and the fact
that the map $t \mapsto \underline{S}^{(h)}(u^{(h)}_t)$ is decreasing, we obtain
the inequalities
$$
\frac{M}{2} \leq -\underline{S}^{(h)}(u^{(h)}_{t_0}) \leq
\frac{1}{p+1}\|u^{(h)}_{t_0}\|_{L^{p+1}}^{p+1} + h\|u^{(h)}_{t_0}\|_{L^{p+1}}^p
+ h\|u^{(h)}_{t_0}\|_{L^{p+1}}
$$ from which we deduce that by taking $M$ sufficiently large one can force
$\|u^{(h)}_{t_0}\|_{L^{p+1}}$ to be large enough so as to guarantee that
$$
\left(\frac{p-1}{p+1}\right) \|u^{(h)}_{t_0}\|_{L^{p+1}}^{p+1} - h
\left(\frac{p+2}{p}\right)\|u^{(h)}_{t_0}\|_{L^{p+1}}^p -
h\|u^{(h)}_{t_0}\|_{L^{p+1}} \geq \frac{1}{2}\left(\frac{p-1}{p+1}\right)
\|u^{(h)}_{t_0}\|_{L^{p+1}}^{p+1}
$$ is satisfied for any $0 \leq h < 1$. Therefore, we see that if $M$
sufficiently large then for all $0 \leq h < 1$ the application $\phi^{(h),u}$
satisfies
\begin{equation}\label{eqpoten3}
\frac{d\phi^{(h),u}}{dt}(t_0) \geq 2M + \left(\frac{p-1}{p+1}\right)
\left(\phi^{(h),u}(t_0)\right)^{\frac{p+1}{2}}
\end{equation} for every $t_0 \in (0,\tau^{(h),u})$, where to obtain
\eqref{eqpoten3} we have once again used Hölder's inequality and the fact that
the map $t \mapsto \underline{S}^{(h)}(u^{(h)}_t)$ is decreasing. Now, it is straightforward to show that the solution $y$ of the ordinary differential equation
$$
\left\{\begin{array}{l} \dot{y} = 2M + \left(\frac{p-1}{p+1}\right)
y^{\frac{p+1}{2}} \\ y(0) \geq 0 \end{array}\right.
$$ explodes before time $$
T = \frac{\delta}{4} +
\frac{2^{\frac{p+1}{2}}(p+1)}{(p-1)^2(M\delta)^{\frac{p-1}{2}}}.
$$ Indeed, either $y$ explodes before time $\frac{\delta}{4}$ or $\tilde{y}:= y(
\cdot + \frac{\delta}{4})$ satisfies
$$
\left\{\begin{array}{l} \dot{\tilde{y}} \geq \left(\frac{p-1}{p+1}\right)
\tilde{y}^{\frac{p+1}{2}} \\ \tilde{y}(0) \geq \frac{M\delta}{2}
\end{array}\right.
$$ which can be seen to explode before time
$$
\tilde{T}=\frac{2^{\frac{p+1}{2}}(p+1)}{(p-1)^2(M\delta)^{\frac{p-1}{2}}}
$$ by performing the standard integration method. If $M$ is taken sufficiently
large then $T$ can be made strictly smaller than $\frac{\delta}{2}$ which, by
\eqref{eqpoten3}, implies that $\tau^{(h),u} < \frac{\delta}{2}$ as desired.

Now let us show statement (ii). Given $K > 0$ let us take $M > 0$ as above and
consider $u \in C_D([0,1])$ satisfying $S(u) \leq -M$ and $\|u\|_\infty \leq K$.
Using Propositions \ref{S.1} and \ref{Lyapunov} adapted to the system
\eqref{randomPDE2} we may find $\rho_{M,K} > 0$ sufficiently small so as to
guarantee that for some small $0 < t_{u} < \frac{\delta}{2}$ any $v \in
B_{\rho_{M,K}}(u)$ satisfies
$$
\underline{S}^{(h)}(\underline{Z}^{(h),v}(t_{u},\cdot)) \leq
\underline{S}^{(h)}(u) +\frac{M}{4}
$$ for all $0 \leq h < 1$. Notice that this is possible since the constants
in \mbox{Proposition \ref{S.1}} adapted to this context can be taken independent from
$h$ provided that $h$ remains bounded. These constants still depend on $\| u
\|_\infty$ though, so that the choice of $\rho_{M,K}$ will inevitably depend on
both $M$ and $K$. Next, let us take $0 < h_{M,K} < 1$ so as to guarantee that
$\underline{S}^{(h)}(u) \leq - \frac{3M}{4}$ for every $0 \leq h < h_{M,K}$.
Notice that, since $\underline{S}^{(h)}(u) \leq S(u) + h(K^p + K),$ it is
possible to choose $h_{M,K}$ depending only on $M$ and $K$. Thus, for any $v \in
B_{\rho_{M,K}}(u)$ we obtain
$\underline{S}^{(h)}(\underline{Z}^{(h),v}(t_{u},\cdot)) \leq - \frac{M}{2}$
which, by the choice of $M$, implies that $\tau^{(h),v} < t_u + \frac{\delta}{2}
< \delta$. This concludes the proof.
\end{proof}

Let us observe that the system $\overline{Z}^{(0),u}$ coincides with $U^u$ for
every $u \in C_D([0,1])$. Thus, by the previous lemma we obtain the following
corollary.

\begin{cor}\label{contdetexp} The application $u \mapsto \tau^u$ is continuous
on $\mathcal{D}_e$.
\end{cor}

\begin{proof} Given $u \in \mathcal{D}_e$ and $\delta > 0$ we show that there
exists $\rho > 0$ such that for all $v \in B_\rho(u)$ we have
$$
-\delta + \tau^u < \tau^v < \tau^u + \delta.
$$ To see this we first notice that by Proposition \ref{G.2} there exists
$\rho_1 > 0$ such that $-\delta + \tau^u < \tau^v$ for any $v \in
B_{\rho_1}(u)$. Moreover, by (i) in Lemma \ref{expestimate} we may take
$M, \tilde{\rho_2} > 0$ such that $\tau^{\tilde{v}} < \delta$ for any $\tilde{v}
\in B_{\tilde{\rho_2}}(\tilde{u})$ with $\tilde{u} \in C_D([0,1])$ verifying
$S(\tilde{u}) \leq - M$. For any such $M$, by Proposition \ref{caract} we
may find some $0 <t_M < t^u$ such that $S( U^u(t_M,\cdot) ) \leq -M$ and using
Proposition \ref{G.2} we may take $\rho_2 > 0$ such that $U^v(t_M,\cdot) \in
B_{\tilde{\rho_2}}(U^u(t_M,\cdot))$ for any $v \in B_{\rho_2}(u)$. This implies
that $\tau^v < t_M + \delta < t^u + \delta$ for all $v \in B_{\rho_2}(u)$ and
thus by taking $\rho = \min\{ \rho_1,\rho_2\}$ we obtain the result.
\end{proof}

The following two lemmas provide the necessary tools to obtain the uniformity in
the upper bound claimed in Theorem \ref{contexp}.

\begin{lema}\label{lemacontsup} Given $M > 0$ and $u \in \mathcal{D}_e$ let us
define the quantities
$$
\mathcal{T}_M^u = \inf\{ t \in [0,\tau^u) : S( U^u(t,\cdot) ) < - M \}
\hspace{1cm}\text{ and }\hspace{1cm}\mathcal{R}_M^u = \sup_{0 \leq t \leq
\mathcal{T}_M^u} \| U^u(t,\cdot)\|_\infty.
$$ Then the applications $u \mapsto \mathcal{T}_M^u$ and $u \mapsto
\mathcal{R}_M^u$ are both upper semicontinuous on $\mathcal{D}_e$.
\end{lema}

\begin{proof} We must see that the sets $\{ \mathcal{T}_M < \alpha\}$ and $\{
\mathcal{R}_M < \alpha \}$ are open in $\mathcal{D}_e$ for all $\alpha > 0$.
But the fact that $\{ \mathcal{T}_M < \alpha\}$ is open follows at once from
Proposition \ref{S.1} and $\{ \mathcal{R}_M < \alpha\}$ is open by Proposition
\ref{G.2}.
\end{proof}

\begin{lema}\label{lemacontsup2} For each $u \in \mathcal{D}_e^+$ let us define
the quantity
$$
\mathcal{I}^u:= \inf_{(t,x) \in [0,\tau^u) \times [0,1]}  U^u(t,x).
$$ Then the application  $u \mapsto I^u$ is lower semicontinuous on
$\mathcal{D}_e^+$.
\end{lema}

\begin{proof} Notice that $\mathcal{I}^u \geq 0$ for any $u \in \mathcal{D}_e^+$
since $U^u(t,0)=U^u(t,1)=0$ for all $t \in [0,\tau^u)$. Therefore, it will
suffice to show that the sets $\{ \alpha < \mathcal{I} \}$ are open in
$\mathcal{D}^+_e$ for every $\alpha < 0$. With this purpose in mind, given
$\alpha < 0$ and $u \in \mathcal{D}_e^+$ such that $\alpha < \mathcal{I}^u$,
take $\beta_1,\beta_2 < 0$ such that $\alpha < \beta_1 < \beta_2 <
\mathcal{I}^u$ and let $y$ be the solution to the ordinary differential equation
\begin{equation}\label{contsup2eq}
\left\{\begin{array}{l} \dot{y} = - |y|^p \\ y(0) = \beta_2.\end{array}\right.
\end{equation} Define $t_\beta := \inf \{ t \in [0,t_{max}^y) : y(t) < \beta_1
\}$, where $t_{max}^y$ denotes the explosion time of $y$. Notice that by the
lower semicontinuity of $S$ for any $M > 0$ we have
$S( U^u( \mathcal{T}^u_M, \cdot) ) \leq -M$ and thus, by Lemma
\ref{expestimate}, we may choose $M$ such that
\begin{equation}\label{contsupeq4}
\sup_{v \in B_\rho( U^u( \mathcal{T}^u_M,\cdot) )} \tau^v < t_\beta
\end{equation} for some small $\rho > 0$. Moreover, if $\rho < \mathcal{I}^u -
\beta_2$ then every $v \in B_\rho( U^u( \mathcal{T}^u_M,\cdot) )$ satisfies
$\inf_{x \in [0,1]} v(x) \geq \beta_2$ so that $U^v$ is in fact a supersolution
to the equation \eqref{contsup2eq}. By \eqref{contsupeq4} this implies that $v
\in \mathcal{D}_e^+$ and $\mathcal{I}^v \geq \beta_1 > \alpha$.
On the other hand, by Proposition \ref{G.2} \mbox{we may} take $\delta > 0$
sufficiently small so that for every $w \in B_\delta(u)$ we have
$\mathcal{T}^u_M < \tau^w$ and
$$
\sup_{t \in [0,\mathcal{T}^u_M]} \| U^w(t,\cdot) - U^u(t,\cdot)\|_\infty < \rho.
$$ Combined with the previous argument, this yields the inclusion $B_\delta(u)
\subseteq \mathcal{D}_e^+ \cap \{ \alpha < \mathcal{I} \}$. In particular, this
shows that $\{ \alpha < \mathcal{I} \}$ is open and thus concludes the proof.
\end{proof}

\begin{obs} The preceding proof shows, in particular, that the set
$\mathcal{D}_e^+$ is open.
\end{obs}
The conclusion of the proof of Theorem \ref{contexp} is contained in the next
proposition.

\begin{prop}\label{convsup}
For any bounded set $\mathcal{K} \subseteq \mathcal{D}_e^*$ at a positive
distance from $\p \mathcal{D}^*_e$ and $\delta > 0$ there exists a constant $C >
0$ such that
\begin{equation}\label{convergenciasuperior}
\sup_{u \in \mathcal{K}} P_u ( \tau_\varepsilon > \tau + \delta ) \leq e^{-
\frac{C}{\varepsilon^2}}.
\end{equation}
\end{prop}

\begin{proof} Since $\mathcal{D}_e^- = - \mathcal{D}_e^+$ and $U^{-u}= - U^u$
for $u \in C_D([0,1])$, without loss of generality \mbox{we may assume} that
$\mathcal{K}$ is contained in $\mathcal{D}_e^+$.
Let us begin by noticing that for any $M > 0$
$$
\mathcal{T}_M := \sup_{u \in \mathcal{K}} \mathcal{T}_M^u < +\infty
\hspace{1cm}\text{ and }\hspace{1cm}\mathcal{R}_M:= \sup_{u \in \mathcal{K}}
\mathcal{R}^u_M < +\infty.
$$ Indeed, by Propositions \ref{G.1} and \ref{A.2} we may choose $t_0 > 0$ small
so as to guarantee that the orbits $\{ U^{u}(t,\cdot) : 0 \leq t \leq t_0 , u
\in \mathcal{K} \}$ remain uniformly bounded and the family $\{ U^{u}(t_0,\cdot)
: u \in \mathcal{K} \}$ is contained in a compact set $\mathcal{K}' \subseteq
\mathcal{D}_e^+$ at a positive distance \mbox{from $\p \mathcal{D}^+_e$.} But
then we have
$$
\mathcal{T}_M \leq t_0 + \sup_{u \in \mathcal{K}'} \mathcal{T}_M^u
\hspace{1cm}\text{ and }\hspace{1cm}\mathcal{R}_M \leq \sup_{0 \leq t \leq t_0,
u \in \mathcal{K}} \| U^u(t,\cdot) \|_\infty + \sup_{u \in \mathcal{K}'}
\mathcal{R}^u_M $$ and both right hand sides are finite due to Lemma
\ref{lemacontsup} and the fact that $\mathcal{T}_M^u$ and $\mathcal{R}_M$ are
both finite for each $u \in \mathcal{D}_e$ by Proposition \ref{caract}.
Similarly, by Lemma \ref{lemacontsup2} we also have
$$
\mathcal{I}_{\mathcal{K}}:= \inf_{u \in \mathcal{K}} \mathcal{I}^u > - \infty.
$$
Now, for each $u \in \mathcal{K}$ and $\varepsilon > 0$ by the Markov property
we have for any $\rho > 0$
\begin{equation}\label{descompexp}
P_u ( \tau_\varepsilon > \tau + \delta ) \leq P( d_{\mathcal{T}_M}(
U^{(\mathcal{R}_M+1),u,\varepsilon}, U^{(\mathcal{R}_M+1),u}) > \rho ) + \sup_{v
\in B_{\rho}(U^u( \mathcal{T}^u_M, \cdot))} P_v ( \tau_\varepsilon > \delta).
\end{equation} The first term on the right hand side is taken care of by
\eqref{grandes1} so that in order to show \eqref{convergenciasuperior} it only
remains to deal with the second term by choosing $M$ and $\rho$ appropriately.
The argument given to deal with this term is similar to that of the proof of
\mbox{Lemma \ref{lemacontsup2}.} Let $y$ be the solution to the ordinary
differential equation
\begin{equation}\label{convsupeq2}
\left\{\begin{array}{l} \dot{y} = - |y|^p - |y|^{p-1} - 1\\ y(0) =
\mathcal{I}_{\mathcal{K}} - \frac{1}{2}.\end{array}\right.
\end{equation} Define $t_{\mathcal{I}} := \inf \{ t \in [0,t_{max}^y) : y(t) <
\mathcal{I}_{\mathcal{K}} - 1 \}$, where $t_{max}^y$ denotes \mbox{the explosion
time of $y$.}
By Lemma \ref{expestimate}, we may choose $M$ such that
\begin{equation}\label{convsupeq4}
\sup_{v \in B_{\rho_M}( U^u( \mathcal{T}^u_M,\cdot) )} \tau^{(h),v} <
\min\{\delta, t_\mathcal{I}\}
\end{equation} for all $0 \leq h < h_M$, where $\rho_M > 0$ and $h_M > 0$ are
suitable constants. The key observation here is that, since $\mathcal{R}_M <
+\infty$, we may choose these constants so as not to depend on $u$ but rather on
$M$ and $\mathcal{R}_M$ themselves. Moreover, if $\rho_M < \frac{1}{2}$ then
every $v \in B_{\rho_M}( U^u( \mathcal{T}^u_M,\cdot) )$ satisfies $\inf_{x \in
[0,1]} v(x) \geq \mathcal{I}_{\mathcal{K}} - \frac{1}{2}$ so that
$\underline{Z}^{(h),v}$ is in fact a supersolution to the equation
\eqref{convsupeq2} for all $0 \leq h < \min\{h_M,1\}$. By \eqref{convsupeq4} the
former implies that $\underline{Z}^{(h),v}$ explodes through $+\infty$ and that
it remains bounded from below by $\mathcal{I}_{\mathcal{K}} - 1$ until its
explosion time which, by \eqref{convsupeq4}, is smaller than $\delta$. In
particular, we see that if $\|V^{\mathbf{0},\varepsilon}\|_\infty < \min \{ 1,
\frac{h_M}{p2^{p-1}}\}$ then $Z^{v,\varepsilon}$ explodes before
$\underline{Z}^{(h),v}$ does, so that we
have that $\tau_\varepsilon < \delta$ under such conditions. Hence, we conclude
that
$$
\sup_{v \in B_{\rho_M}(U^u( \mathcal{T}^u_M, \cdot))} P_v ( \tau_\varepsilon >
\delta) \leq P \left( \sup_{t \in [0,\delta]}
\|V^{\mathbf{0},\varepsilon}(t,\cdot) \|_\infty \leq \min \left\{ 1,
\frac{h_M}{p2^{p-1}} \right\} \right)
$$ which, by recalling the estimate \eqref{grandes1}, gives the desired control
on the second term in the right hand side of \eqref{descompexp}. Thus, by taking
$\rho$ equal to $\rho_M$ in \eqref{descompexp}, we obtain the result.

\end{proof}

This last proposition in fact shows that for $\delta > 0$ and a given bounded
set $\mathcal{K} \subseteq \mathcal{D}_e^*$ at a positive distance from $\p
\mathcal{D}^*_e$ there exist constants $M,C > 0$ such that
$$
\sup_{u \in \mathcal{K}} P_u ( \tau_\varepsilon > \mathcal{T}_M^u + \delta )
\leq e^{- \frac{C}{\varepsilon^2}}.
$$ By exploiting the fact $\mathcal{T}_M < +\infty$ for every $M > 0$ we obtain
the following useful corollary.

\begin{cor}\label{exploacot} For any bounded $\mathcal{K} \subseteq
\mathcal{D}_e^*$ at a positive distance from $\p \mathcal{D}_e^*$ there exist
constants $\tau_K, C > 0$ such that
$$
\sup_{u \in \mathcal{K}} P_u ( \tau_\varepsilon > \tau_K) \leq e^{-
\frac{C}{\varepsilon^2}}.
$$
\end{cor}

\section{Construction of an auxiliary domain}\label{secg}

To study the behavior of the explosion time for initial data in
$\mathcal{D}_\mathbf{0}$ it is convenient to introduce an auxiliary bounded
domain $G \subseteq C_D([0,1])$ containing a neighborhood $B_c$ of the stable
equilibrium and such that for any initial data $u \in B_c$ the escape time
from this domain is asymptotically equivalent to the explosion time. By doing so
we can then reduce our original problem to a simpler one: characterizing the
escape from this domain. This becomes a simpler problem because, since the
escape only depends on the behavior of the system while it remains inside a
bounded region, local large deviation estimates can be successfully applied to
its study. This approach is not new, it was originally proposed in \cite{GOV} to
study the finite-dimensional double-well potential model. However, in our
present setting the construction of this auxiliary domain is much more involved
and, as a matter of fact, a priori it is not even clear that such a domain
exists for every value
of $p > 1$. The aim of this section is to construct such a domain for $1 < p < 5$. The following lemma will
play a key role in this.

\begin{lema}[\cite{S}]\label{compacidad} If $1 < p < 5$ then the set $\{
u \in \overline{\mathcal{D}_{\mathbf{0}}} : S(u) \leq a \}$ is bounded in $C([0,1])$ for any $a > 0$.
\end{lema}

\begin{proof} For $a > 0$ and $v \in \{ u \in
\overline{\mathcal{D}_{\mathbf{0}}} : S(u) \leq a \}$ consider $\psi : \R_{\geq
0} \rightarrow \R_{\geq 0}$ given by
$$
\psi(t):= \int_0^1 (U^v(t,\cdot))^2.
$$ A direct computation shows that for every $t_0 > 0$ the function $\psi$
satisfies
$$
\frac{d\psi(t_0)}{dt} = - 4 S( U^v(t,\cdot) ) + 2\left(\frac{p-1}{p+1}\right)
\int_0^1 |U^v(t,\cdot)|^{p+1}.
$$ By Proposition \ref{Lyapunov} and Hölder's inequality we then obtain
$$
\frac{d\psi(t_0)}{dt} \geq - 4 a + 2\left(\frac{p-1}{p+1}\right)
(\psi(t_0))^{\frac{p+1}{2}}
$$ which implies that $\psi(0) \leq B:=
\left[2a\left(\frac{p+1}{p-1}\right)\right]^{\frac{2}{p+1}}$ since otherwise
$\psi$ (and therefore $U^v$) would explode in finite time. Now, by the
Gagliardo-Niremberg interpolation inequality (recall that $v$ is absolutely
continuous since $S(v) < +\infty$)
$$
\| v \|_\infty^2 \leq C_{GN} \| v \|_{L^2} \| \p_x v \|_{L^2},
$$ we obtain
$$
\int_0^1 |v|^{p+1} \leq \| v \|_{L^2}^2 \|v \|_{\infty}^{p-1} \leq C_{GN}^{\frac{p-1}{2}}
B^{\frac{p+3}{4}} \| \p_x v \|_{L^2}^{\frac{p-1}{2}} \leq C_{GN}^{\frac{p-1}{2}}
B^{\frac{p+3}{4}} (2a + \int_0^1 |v|^{p+1})^{\frac{p-1}{4}}
$$ which for $p < 5$ implies the bound
\begin{equation}\label{bound1}
\int_0^1 |v|^{p+1} \leq B':=\max \left\{ 2a , \left[C_{GN}^{\frac{p-1}{2}}
B^{\frac{p+3}{4}}2^{\frac{p-1}{4}}\right]^{\frac{4}{5-p}}\right\}.
\end{equation} Since $S(v) \leq a$ we see that \eqref{bound1} implies the bound
$\| \p_x v \|_{L^2} \leq \sqrt{2B'}$. Thus, we conclude
$$
\| v \|_\infty \leq \left( C_{GN}^{p-1} 2BB'\right)^{\frac 1 4}
$$ which shows that $\{ u \in \overline{\mathcal{D}_{\mathbf{0}}} : 0 \leq S(u)
\leq a \}$ is bounded.
\end{proof}

\begin{obs} The proof of Lemma \ref{compacidad} is the only instance throughout
our work in which the assumption $p < 5$ is used. As a matter of fact, we only require the
weaker condition that there exists $\alpha > 0$ such that the set $\{ u \in
\overline{\mathcal{D}_{\mathbf{0}}} : S(u) \leq S(z) + \alpha \}$ is bounded.
However, determining the validity of this condition for arbitrary $p > 1$ does
not seem simple.
\end{obs}

Before we can carry on with the next proposition, we need to introduce some definitions.

\begin{defi} Given $T > 0$ and $\varphi \in C_{D}([0,T] \times [0,1])$ we define
the \textit{rate} $I(\varphi)$ of $\varphi$ by the formula
$$
I(\varphi) := I^{\varphi(0,\cdot)}_T(\varphi),
$$ where $I^{\varphi(0,\cdot)}_T$ is defined as in Section \ref{secLDP}.
\end{defi}

\begin{defi} We say that a function $\varphi \in C_D([0,T]\times [0,1])$ is
\textit{regular} if both derivatives $\p_t \varphi$ and $\p^2_{xx} \varphi$
exist and belong to $C_D ([0,T]\times [0,1])$.
\end{defi}

\begin{prop}\label{costo} Given $T > 0$, for any $\varphi \in C_{D} \cap
W^{1,2}_2 ([0,T] \times [0,1])$ such that $\p^2_{xx}\varphi(0,\cdot)$ exists and
belongs to $C_D([0,1])$ we have that
\begin{equation}\label{cotainftasa}
I(\varphi) \geq 2 \left[\sup_{0 \leq T' \leq T} \left(
S(\varphi(T',\cdot))-S(\varphi(0,\cdot))\right)\right].
\end{equation}
\end{prop}

\begin{proof}Assume first that $\varphi$ is regular. Using that $(x-y)^2 =
(x+y)^2 -4xy$ for $x,y \in \R$, for any $0 \leq T' \leq T$ we obtain that
\begin{align*}
I(\varphi) & = \frac{1}{2} \int_0^T \int_0^1 |\p_t \varphi - \p_{xx}^2 \varphi -
g(\varphi)|^2 \geq \frac{1}{2} \int_0^{T'} \int_0^1 |\p_t \varphi - \p_{xx}^2
\varphi - g(\varphi)|^2 \\
\\
& = \frac{1}{2} \int_0^{T'} \int_0^1 \left[|\p_t \varphi + \p_{xx}^2 \varphi +
g(\varphi)|^2 - 4\left( \p_{xx}^2 \varphi + g(\varphi)\right)\p_t
\varphi\right]\\
\\
& =  \frac{1}{2} \int_0^{T'} \left[ \left( \int_0^1 |\p_t \varphi + \p_{xx}^2
\varphi + g(\varphi)|^2\right) + 4 \frac{dS( \varphi(t,\cdot) )}{dt}\right]\\
\\
& \geq 2\left( S(\varphi(T',\cdot))-S(\varphi(0,\cdot))\right).
\end{align*} Taking supremum on $T'$ yields the result in this particular case.
Now, if $\varphi$ is not necessarily regular then by \cite[Theorem~6.9]{FJL} we
may take a sequence $(\varphi_n)_{n \in \N}$ of regular functions converging to
$\varphi$ on $C_{D_{\varphi(0,\cdot)}}([0,T]\times[0,1])$ and such that $\lim_{n
\rightarrow +\infty} I(\varphi_n) = I(\varphi)$ is satisfied. The result in the
general case then follows from the validity of \eqref{cotainftasa} for regular
functions and the lower semicontinuity of $S$.
\end{proof}

In order to properly interpret the content of Proposition \ref{costo} we need to
introduce the concept of \textit{quasipotential} for our system. We do so in the
following definitions.

\begin{defi} Given $u,v \in C_D([0,1])$ a \textit{path from $u$ to $v$} is a
continuous function $\varphi \in C_{D}([0,T] \times [0,1])$ for some $T > 0$
such that $\varphi(0,\cdot)=u$ and $\varphi(T,\cdot)=v$.
\end{defi}

\begin{defi} Given $u,v \in C_D([0,1])$ we define the \textit{quasipotential}
$V(u,v)$ \mbox{from $u$ to $v$} by the formula
$$
V(u,v)= \inf \{ I(\varphi) : \varphi \text{ path from $u$ to $v$}\}.
$$ Furthermore, given a subset $B \subseteq C_D([0,1])$ we define the
quasipotential from $u$ to $B$ as
$$
V(u,B):= \inf \{ V(u,v) : v \in B \}.
$$ We refer the reader to the appendix for a review of the properties of $V$ we shall use.
\end{defi}

In a limiting sense, made rigorous through the large deviations estimates in
Section \ref{secLDP}, the quasipotential $V(u,v)$ represents the energy cost for
the stochastic system to travel from $u$ to (an arbitrarily small neighborhood
of) $v$. Notice that Lemma \ref{compacidad} implies that $\lim_{n \rightarrow
+\infty} V(\mathbf{0},\partial B_{n} \cap \mathcal{D}_{\mathbf{0}})=+\infty$,
which says that the energy cost for the stochastic system starting from
$\mathbf{0}$ to explode in a finite time while remaining inside
$\mathcal{D}_\mathbf{0}$ is infinite. Thus, should explosion occur, it would
involve the system stepping outside $\mathcal{D}_{\mathbf{0}}$ and crossing
$\mathcal{W}$. In view of Proposition \ref{costo}, the crossing of $\mathcal{W}$
will typical take place through $\pm z$ since the energy cost for performing
such a feat is the lowest there. Therefore, if we wish the escape from $G$ to
capture the essential characteristics of the explosion phenomenon in the
stochastic system (at least when
starting from $\mathbf{0}$) then it is important to guarantee that this escape
involves passing through (an arbitrarily small neighborhood of) $\pm z$. Not
only this, but we also require that once the system escapes this domain then it
explodes with overwhelming probability in a quick fashion, i.e. before some time
$\tau^*$ which does not depend on $\varepsilon$. More precisely, we wish to
consider a bounded domain $G \subseteq C_D([0,1])$ verifying the following
properties:
\begin{cond}\label{assumpg}$\,$
\begin{enumerate}
\item [i.] There exists $r_{\mathbf{0}}>0$ such that $B_{2r_\mathbf{0}}
\subseteq \mathcal{D}_{\mathbf{0}} \cap G$.
\item [ii.] There exists $c > 0$ such that $B_c \subseteq B_{r_\mathbf{0}}$ and
for all
$v \in B_c$ the solution $U^{v}$ to \eqref{MainPDE} with initial datum $v$ is
globally defined and converges to $\mathbf{0}$ without
escaping $B_{r_\mathbf{0}}$.
\item [iii.] There exists a closed subset $\p^{\pm z}$ of the boundary $\partial
G$ which satisfies
\begin{enumerate}
\item [$\bullet$] $V(\mathbf{0},\partial G - \partial^{\pm z} ) >
V(\mathbf{0},\partial^{\pm z}) = V( \mathbf{0}, \pm z)$.
\item [$\bullet \bullet$] $\p^{\pm z}$ is contained in $\mathcal{D}_e^*$ and at
a positive distance from its boundary.
\end{enumerate}
\end{enumerate}
\end{cond}

In principle, we have seen that such a domain is useful to study the behavior of
the explosion time whenever the initial datum of the stochastic system is (close
to) the origin. Nevertheless, by the local estimate \eqref{grandes1}, when
starting inside $\mathcal{D}_\mathbf{0}$ the system will typically visit a small
neighborhood of the origin before crossing $\mathcal{W}$ and thus such a choice
of $G$ will also be suitable to study the explosion time for arbitrary initial
data in $\mathcal{D}_\mathbf{0}$.

The construction of the domain $G$ is done as follows. Since
$\mathcal{D}_{\mathbf{0}}$ is open we may choose $r_{\mathbf{0}} > 0$ such that
$B_{3r_{\mathbf{0}}}$ is contained in $\mathcal{D}_{\mathbf{0}}$. Moreover, by
the asymptotic stability of $\mathbf{0}$ we may choose $c > 0$ verifying (ii) in
Conditions \ref{assumpg}. Now, given $\zeta_1 > 0$ by Lemma \ref{compacidad} we
may take $n_0 \in \N$ such that $n_0 > 3r_{\mathbf{0}}$ and the set $\{ u \in
\overline{\mathcal{D}_{\mathbf{0}}} : S(u) \leq S(z) + \zeta_1 \}$ is contained
in the interior of the ball $B_{n_0-1}$. We then define the pre-domain
$\tilde{G}$ as
\begin{equation}\label{predomain}
\tilde{G}:= B_{n_0} \cap \overline{\mathcal{D}_{\mathbf{0}}}.
\end{equation}
Notice that since both $B_{n_0}$ and $\overline{\mathcal{D}_\mathbf{0}}$ are
closed sets we have that
$$
\partial \tilde{G} = \left( \mathcal{W} \cap B_{n_0}\right) \cup \left( \partial
B_{n_0} \cap \mathcal{D}_\mathbf{0}\right)
$$ which, by the particular choice of $n_0$ and Proposition \ref{Lyapunov},
implies $\min_{u \in \p \tilde{G}} S(u) =S(z)$. By Propositions \ref{costo} and
\ref{A.4} we thus obtain $V(\mathbf{0},\p \tilde{G}) \geq \Delta$. Next, if for
$u \in C_D([0,1])$ we let $u^-$ denote the negative part of $u$, i.e. $u^- =
\max \{ - u, 0 \}$, then since $z^- = \mathbf{0}$ we may find $\tilde{r}_z > 0$
such that $u^- \in \mathcal{D}_{\mathbf{0}}$ for any $u \in B_{\tilde{r}_z}(z)$.
Finally, if for $r > 0$ we write $B_{r}(\pm z) := B_{r}(z) \cup B_{r}(-z)$ and
take $r_{z} > 0$ such that $r_z < \frac{\tilde{r}_z}{2}$, $B_{2r_z}(\pm z)$ is
contained in the interior of $B_{n_0}$ and $z$ is the unique equilibrium point
of the system lying inside $B_{r_z}(z)$, then we define our final domain $G$ as
$$
G= \tilde{G} \cup B_{r_z}(\pm z).
$$ Let us now check that this domain satisfies all the required conditions. We
begin by noticing that (i) and (ii) in Conditions \ref{assumpg} are immediately
satisfied by the \mbox{choice of $n_0$.}
Now, let us also observe that for any $r > 0$
\begin{equation}\label{lejosdelminimo}
\inf\{ S(u) : u \in \p \tilde{G} - B_{r}(\pm z)\} > S(z).
\end{equation} Indeed, if this is not the case then, since we have $S(z)=\inf_{u
\in \mathcal{W}} S(u)$ by Corollary \ref{corzminimo}, there exists a sequence
$(u_k)_{k \in \N} \subseteq \left[\mathcal{W}\cap B_{n_0} - B_{r}(\pm z)\right]$
such that $\lim_{k \rightarrow +\infty} S(u_k) = S(z)$. Then, \mbox{by
Proposition \ref{G.1}} we have that there exists $t_0 > 0$ sufficiently small
satisfying
$$
\sup_{k \in \N} \left[\sup_{t \in [0,t_0]} \| U^{u_k}(t,\cdot) \|_\infty\right]
< +\infty \hspace{1cm}\text{ and }\hspace{1cm} \inf_{k \in \N} \|
U^{u_k}(t_0,\cdot) - (\pm z) \|_\infty > \frac{r}{2}
$$ and therefore by Proposition \ref{A.2} we may conclude that there exists a
subsequence $(u_{k_j})_{j \in \N}$ such that $U^{u_{k_j}}(t_0,\cdot)$ converges
to a limit $u_{\infty} \in C_D([0,1])$ as $j \rightarrow +\infty$. Since the
potential is lower semicontinuous and $\mathcal{W}$ is both closed and invariant
under the deterministic flow, by Proposition \ref{Lyapunov} we conclude that
$u_\infty = \pm z$ which contradicts the fact that the sequence
$(U^{u_{k_j}}(t_0,\cdot))_{j \in \N}$ is at a positive distance from these
equilibriums. Hence, we obtain \eqref{lejosdelminimo}.
In particular, this implies that $V(\mathbf{0}, \p \tilde{G} - B_{r}(\pm z)) >
\Delta$ is satisfied for any choice of $r > 0$. Let us then take $\zeta_2 > 0$
such that $\Delta + \zeta_2 < V(\mathbf{0},\p \tilde{G} - B_{\frac{r_z}{2}}(\pm
z))$ and define
$$
\tilde{\p}^z:= \{ u \in \p B_{r_z}(z) \cap \overline{\mathcal{D}_e} :
V(\mathbf{0},u) \leq \Delta + \zeta_2 \}.
$$
Notice that $\tilde{\p}^z$ is nonempty and satisfies $d(\tilde{\p}^z,
\mathcal{W}) > 0$. Indeed, it is possible to show that for each $\alpha > 0$
there exists a path from $\mathbf{0}$ to $B_{r_z}(z) \cap
\overline{\mathcal{D}_e}$ with rate function less than $\Delta + \alpha$, which
immediately implies that $\tilde{\p}^z$ is nonempty. This path is essentially
obtained by going from $\mathbf{0}$ to $z$ by describing the orbit given by the
unstable manifold $\mathcal{W}_u^z$ in reverse order, then making a linear
interpolation towards $(1+h)z$ for some $h \in (0,r_z)$ sufficiently small and
ultimately following the deterministic flow until it reaches $B_{r_z}(z)$
(notice that this will eventually happen since $(1+h)z \in \mathcal{D}_e$ by
Proposition \ref{descomp2} and $\mathcal{D}_e$ is invariant). We refer to
\cite[Lemma~4.3]{SJS} for details on the construction. On the
other hand, if $d(\tilde{\p}^z, \mathcal{W}) = 0$ then we can construct
sequences $(u_k)_{k \in \N} \subseteq \mathcal{W}$ and
$(v_k)_{k \in \N} \subseteq \tilde{\p}^z$ such that $\lim_{k \rightarrow
+\infty} d(u_k,v_k)=0$. The growth estimates on the appendix imply that there
exists $t_1 > 0$ sufficiently small such that
$$
\lim_{k \rightarrow +\infty}
d(U^{u_k}(t_1,\cdot),U^{v_k}(t_1,\cdot))=0\hspace{0.2cm}\text{and}\hspace{0.2cm}
\frac{r_z}{2} < \inf_{k \in \N} d(U^{v_k}(t_1,\cdot),z) \leq \sup_{k \in \N}
d(U^{v_k}(t_1,\cdot),z) <  2r_z.
$$ By Proposition \ref{A.2} we obtain that for some appropriate subsequence we
have
$$
\lim_{j \rightarrow +\infty} U^{u_{k_j}}(t_1,\cdot) =\lim_{j \rightarrow
+\infty} U^{v_{k_j}}(t_1,\cdot) = v_\infty.
$$ Observe that $v_\infty \in \mathcal{W} \cap B_{n_0} - B_{\frac{r_z}{2}}(\pm
z)$ and thus that $v_\infty \in \p \tilde{G}- B_{\frac{r_z}{2}}(\pm z)$.
Furthermore, by the lower semicontinuity of $V(\mathbf{0},\cdot)$ and the fact
that the mapping $t \mapsto V(\mathbf{0},U^u(t,\cdot))$ is monotone decreasing
for any $u \in C_D([0,1])$ (see the appendix for details), we obtain that
$V(\mathbf{0},v_\infty) \leq \Delta + \zeta_2$ which, together with the previous
observation, implies the contradiction $\Delta + \zeta_2 \geq V(\mathbf{0},\p
\tilde{G} - B_{\frac{r_z}{2}}(\pm z))$. Hence, we see that $d(\tilde{\p}^z,
\mathcal{W}) > 0$ and thus we may define
$$
\p^z = \left\{ u \in \p B_{r_z}(z) \cap \overline{\mathcal{D}_e} :
d(u,\mathcal{W}) \geq \frac{d(\tilde{\p}^z,\mathcal{W})}{2} \right\}
$$ and set $\p^{\pm z}:= \p^z \cup (-\p^z)$. Since one can easily check that
$$
\p G = [ \p \tilde{G} - B_{r_z}(\pm z) ] \cup [\p B_{r_z}(\pm z) \cap
\overline{D_e} ]
$$ we conclude that $V(\mathbf{0},\p G - \p^{\pm z}) \geq \Delta + \zeta_2$. On
the other hand, by using Proposition \ref{costo} together with the existence of
paths as described above, which go from $\mathbf{0}$ to $\tilde{\p}^z$ by
passing through $z$ and have a rate function which can be made arbitrarily close
to $\Delta$, we get that $V(\mathbf{0},\p^z)=V(\mathbf{0},\tilde{\p}^z)=
V(\mathbf{0},\pm z)=\Delta$ from which one obtains
$$
V(\mathbf{0},\p G - \p^{\pm z}) > V(\mathbf{0},\p^z) = V(\mathbf{0},\pm z).
$$ Furthermore, by the comparison principle and the choice of $\tilde{r}_z$ we
have \mbox{$B_{\tilde{r}_z}(\pm z) \cap \mathcal{D}_e \subseteq
\mathcal{D}_e^*$.} Therefore, since we clearly have $d(\p^{\pm z}, \mathcal{W})
> 0$ by definition of $\p^{\pm z}$, upon recalling that $\p^{\pm z} \subseteq
\mathcal{D}_e$ and $r_z < \frac{\tilde{r}_z}{2}$ we see that $\p^{\pm z}
\subseteq \mathcal{D}_e^+$ and $d(\p^{\pm z}, \p \mathcal{D}_e^*) \geq
\min\{d(\p^{\pm z}, \mathcal{W}), \frac{\tilde{r}_z}{2}\} > 0$, so that
condition (iii) also holds. See Figure \ref{fig2}.

\begin{figure}
	\centering
	\includegraphics[width=8cm]{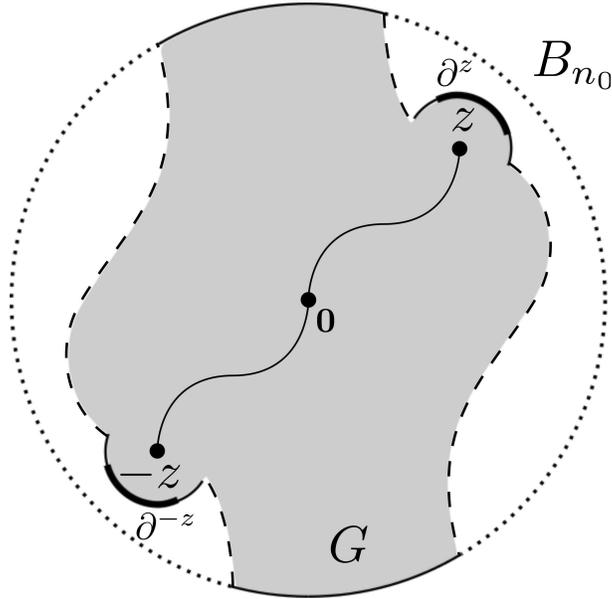}
	\caption{The auxiliary domain $G$}
	\label{fig2}
\end{figure}

\begin{obs}\label{obsequivG} Let us notice that, by Corollary \ref{exploacot},
($\bullet \bullet$) in Conditions \ref{assumpg} implies that there exist
constants $\tau^*,C > 0$ such that
$$
\sup_{u \in \p^{\pm z}} P_u (\tau_\varepsilon > \tau^* ) \leq
e^{-\frac{C}{\varepsilon^2}}
$$ for all $\varepsilon > 0$ sufficiently small. Since ($\bullet$) guarantees
that the escape from $G$ will typically take place through $\p^{\pm z}$, this
tells us that both $\tau_\varepsilon$ and $\tau_\varepsilon(\p G)$ are
asymptotically equivalent, so that it will suffice to study the escape from $G$
in order to establish each of our results.
\end{obs}

\section{The escape from $G$}\label{secescapedeg}

The problem of escaping a bounded domain with similar characteristics to the
ones detailed in Conditions \ref{assumpg} already appears in the literature. In
\cite{GOV,OV}, the authors study the escape from a finite-dimensional domain
containing a stable equilibrium and only one saddle point. Our domain $G$ bears
the additional difficulties of being infinite-dimensional and also of possibly
containing other unstable equilibria besides $\pm z$. On the other hand, in
\cite{B1} they do deal with an infinite-dimensional domain, but this domain has
unstable equilibria only in its boundary and does not contain any of them in its
interior as opposed to what happens in our current situation. Despite the fact
that our domain does not quite fall into any of the cases studied before, all
the results of interest in our present setting can still be obtained by
combining the ideas from these previous works, eventually making some slight
modifications along the way. We outline below the main results regarding the
escape
from the domain $G$ and refer the reader to \cite{SJS} for details on their
proofs. Hereafter, $c > 0$ is taken as in Conditions \ref{assumpg}.

The first result is concerned with the asymptotic order of magnitude of the exit
time.

\begin{theorem}\label{ecotsuplema1}Given $\delta > 0$ we have
$$
\lim_{\varepsilon \rightarrow 0 } \left[\sup_{u \in B_c} \left| P_{u}
\left( e^{\frac{\Delta - \delta}{\varepsilon^{2}}} < \tau_{\varepsilon}(\partial
G) < e^{\frac{\Delta + \delta}{\varepsilon^{2}}}\right)-1 \right|\right] = 0,
$$ where $\tau^u_\varepsilon(\p G):= \inf \{ t > 0 : U^{u,\varepsilon}(t,\cdot)
\in \p G\}$ denotes the exit time from $G$.
\end{theorem}

The second result gives information about the typical escape routes chosen by
$U^\varepsilon$.

\begin{theorem}\label{teoescape0} The stochastic system verifies
$$
\lim_{\varepsilon \rightarrow 0} \left[\sup_{u \in B_c} P_u \left(
U^{\varepsilon}(\tau_\varepsilon(\partial G),\cdot) \notin \partial^{\pm
z}\right)\right] = 0.
$$
Furthermore, if $\tilde{G}$ is the pre-domain constructed in Section \ref{secg},
then for any $\delta > 0$
\begin{equation}\label{ecota0}
\lim_{\varepsilon \rightarrow 0} \left[ \sup_{u \in B_c} P_u \left(
U^\varepsilon \left( \tau_\varepsilon(\p \tilde{G}),\cdot \right) \notin
B_\delta(\pm z) \right) \right] = 0.
\end{equation}
\end{theorem}

The asymptotic distribution of the exit time is established in this third
result.

\begin{theorem}\label{asympteog} For each $\varepsilon > 0$ define the normalization
coefficient $\gamma_\varepsilon > 0$ by the relation
$$
P_{\mathbf{0}}( \tau_\varepsilon(\partial G) > \gamma_\varepsilon ) = e^{-1}.
$$ Then there exists $\rho > 0$ such that for every $t \geq 0$
\begin{equation}\label{convunie}
\lim_{\varepsilon \rightarrow 0} \left[ \sup_{u \in B_\rho} |P_u(
\tau_\varepsilon(\partial G) > t\gamma_\varepsilon) - e^{-t}|\right] = 0.
\end{equation}
\end{theorem}

Finally, the stability of time averages is shown in the forth and last result.

\begin{theorem}\label{stableg} There exists a sequence $(R_\varepsilon)_{\varepsilon
> 0}$ with $\lim_{\varepsilon \rightarrow 0} R_\varepsilon = +\infty$ and
$\lim_{\varepsilon \rightarrow 0} \frac{R_\varepsilon}{\gamma_\varepsilon} = 0$
such that given $\delta > 0$ we have
$$
\lim_{\varepsilon \rightarrow 0} \left[ \sup_{u \in B_c} P_u \left( \sup_{0 \leq
t \leq \tau_\varepsilon(\p G) - 3R_\varepsilon}\left|
\frac{1}{R_\varepsilon}\int_t^{t+R_\varepsilon} f(U^{\varepsilon}(s,\cdot))ds -
f(\mathbf{0})\right| > \delta \right) \right] = 0
$$ for any bounded continuous function $f: C_D([0,1]) \rightarrow \R$.
\end{theorem}

\begin{obs} We would like to point out that the main technical point in the
proof of Theorem \ref{asympteog} is to show that for small $\rho > 0$
\begin{equation}\label{couple}
\lim_{\varepsilon \rightarrow 0}\left[\sup_{u,v \in B_{\rho}} \left[ \sup_{t >
t_0} |P_u(\tau_{\varepsilon}(\partial G)  > t\gamma_{\varepsilon}) -
P_{v}(\tau_{\varepsilon}(\partial G) > t\gamma_{\varepsilon})|\right]\right] =
0.
\end{equation} We do this as in \cite{B2} with the help of the coupling of
solutions with different initial data proposed in \cite{M}. Some technical
difficulties which are not present in \cite{B2} arise in the construction of the
coupling due to the behavior of the source term $g$ but, nonetheless, it is
still possible to couple solutions with initial data sufficiently close to
$\mathbf{0}$ so that \eqref{couple} can be obtained. We refer to \cite{SJS} for
details.
\end{obs}

\section{Asymptotic behavior of $\tau_\varepsilon^u$ for $u \in
\mathcal{D}_{\mathbf{0}}$}\label{secfinal}

We now use the analysis from Section \ref{secescapedeg} to derive our main
results with respect to the metastable behavior of $U^{u,\varepsilon}$ for
initial data $u \in \mathcal{D}_{\mathbf{0}}$. We begin by showing that,
uniformly over bounded sets at a positive distance from $\mathcal{W}$, for $u
\in \mathcal{D}_{\mathbf{0}}$ the system $U^{u,\varepsilon}$ typically visits a
small neighborhood of $\mathbf{0}$ before explosion without ever exiting
$\mathcal{D}_{\mathbf{0}}$.

\begin{lema}\label{taubc}
For any bounded set $\mathcal{K} \subseteq \mathcal{D}_{\mathbf{0}}$ at a
positive distance from $\p\mathcal{D}_{\mathbf{0}}$ and $\rho > 0$
\begin{equation}\label{taubc1}
\lim_{\varepsilon \rightarrow 0} \left[ \sup_{u \in \mathcal{K}} \left| P_u
\left(  \tau_\varepsilon( B_\rho ) <
\min\{\tau_\varepsilon(\mathcal{D}_{\mathbf{0}}^c), \tau_\varepsilon \}\right) -
1 \right| \right] = 0
\end{equation}where $\tau_\varepsilon^u(B_\rho) := \inf \{ t > 0 :
U^{u,\varepsilon}(t,\cdot) \in B_\rho \}$.
\end{lema}

\begin{proof} Let us observe that for any $u \in \mathcal{D}_{\mathbf{0}}$ the
system $U^u$ reaches the set $B_{\frac{\rho}{2}}$ in a finite time
$\tau^u(B_{\frac{\rho}{2}})$ while remaining at all times inside the ball
$B_{r^u}$, where $r^u:= \sup_{t \geq 0} \| U^u(t,\cdot) \|_\infty$, and at a
certain positive distance $d^u:= \inf_{t \geq 0} d(U^u(t,\cdot),\p
\mathcal{D}_{\mathbf{0}})$ from the boundary of $\mathcal{D}_{\mathbf{0}}$.
Therefore, if $\mathcal{K}$ is any bounded set contained in
$\mathcal{D}_{\mathbf{0}}$ at a positive distance from $\p
\mathcal{D}_{\mathbf{0}}$ then we have that the quantities
$\tau_{\mathcal{K},\frac{\rho}{2}} := \sup_{u \in \mathcal{K}}
\tau^u(B_{\frac{\rho}{2}})$ and $r_{\mathcal{K}}:=\sup_{u \in \mathcal{K}} r^u$
are both finite while $d_{\mathcal{K}}:=\inf_{u \in \mathcal{K}} d^u$ is
strictly positive. Indeed, the finiteness of $\tau_{\mathcal{K},\frac{\rho}{2}}$
follows at once from Proposition \ref{A.3} whereas $r_{\mathcal{K}}$ is finite
since by Proposition \ref{G.1}
one may find $t_0 > 0$ sufficiently small such that $\sup_{u \in \mathcal{K}}
\left[ \sup_{t \in [0,t_0]} \| U^u(t,\cdot) \|_\infty \right]$ is finite. That
$\sup_{u \in \mathcal{K}} \left[ \sup_{t \geq t_0} \| U^u(t,\cdot) \|_\infty
\right]$ is finite then follows as in the proof of Proposition \ref{A.3} due to
the fact that the mapping $u \mapsto r^u$ is both upper semicontinuous and
finite on $\mathcal{D}_{\mathbf{0}}$. Finally, the fact that the quantity
$d_{\mathcal{K}}$ is strictly positive follows in a similar manner, using the
fact that the mapping $u \mapsto d^u$ is both lower semicontinuous and positive
on $\mathcal{D}_{\mathbf{0}}$. Now, if we write
$\mathcal{T}^{\mathcal{K}}_\varepsilon:= \tau_{\mathcal{K},\frac{\rho}{2}}
\wedge \tau^{(r_{\mathcal{K}}+1)}_\varepsilon$ then notice that for any $u \in
\mathcal{K}$ we have the bound
\begin{equation}\label{taubc2}
P_u \left( \tau_\varepsilon \leq \tau_\varepsilon (B_\rho) \right) \leq P_u
\left( \tau_\varepsilon(B_\rho) > \tau_{\mathcal{K},\frac{\rho}{2}} \right) +
P_u \left( \tau_\varepsilon \leq \tau_{\mathcal{K},\frac{\rho}{2}}\right) + P_u
\left( \tau_\varepsilon(\mathcal{D}_{\mathbf{0}}^c) \leq
\tau_{\mathcal{K},\frac{\rho}{2}}\right)
\end{equation} with
$$
P_u \left( \tau_\varepsilon(B_\rho) > \tau_{\mathcal{K},\frac{\rho}{2}} \right)
\leq P_u\left( d_{\mathcal{T}^{\mathcal{K}}_\varepsilon}( U^{\varepsilon}, U ) >
\min\left\{\frac{\rho}{2}, \frac{1}{2}\right\}\right)
$$ and
$$
P_u \left( \tau_\varepsilon \leq \tau_{\mathcal{K},\frac{\rho}{2}}\right) \leq
P_u \left( d_{\mathcal{T}^{\mathcal{K}}_\varepsilon}( U^{\varepsilon}, U) >
\frac{1}{2}\right).
$$ As for the third term in the right hand side of \eqref{taubc2}, we have
\begin{align*}
P_u \left( \tau_\varepsilon(\mathcal{D}_{\mathbf{0}}^c) \leq
\tau_{\mathcal{K},\frac{\rho}{2}}\right) & \leq P_u \left(
\tau_\varepsilon^{(r_{\mathcal{K}}+1)} \leq
\tau_{\mathcal{K},\frac{\rho}{2}}\right) + P_u \left(
\tau_\varepsilon(\mathcal{D}_{\mathbf{0}}^c) \leq
\tau_{\mathcal{K},\frac{\rho}{2}} <
\tau_\varepsilon^{(r_{\mathcal{K}}+1)}\right)\\
& \leq P_u \left( d_{\mathcal{T}^{\mathcal{K}}_\varepsilon}( U^{\varepsilon}, U)
> \frac{1}{2}\right) + P_u \left( d_{\mathcal{T}^{\mathcal{K}}_\varepsilon}(
U^{\varepsilon}, U) > \frac{d_{\mathcal{K}}}{2}\right)
\end{align*}
 The uniform bounds given by \eqref{grandes1} now allow us to conclude the
result.
\end{proof}

The next step is to show that, for initial data in a small neighborhood of the
origin, the explosion time and the exit time from $G$ are asymptotically
equivalent.

\begin{lema}\label{nescapelema3} If $\tau^* > 0$ is taken as in Remark
\ref{obsequivG} then
\begin{equation}
\lim_{\varepsilon \rightarrow 0} \left[ \sup_{u \in B_c} P_u ( \tau_\varepsilon
> \tau_\varepsilon (\partial G) + \tau^* )\right] = 0.
\end{equation}
\end{lema}

\begin{proof} For any $u \in B_c$ the strong Markov property implies that
$$
P_u ( \tau_\varepsilon > \tau_\varepsilon (\partial G) + \tau^* ) \leq \sup_{v
\in B_c} P_v \left( U^\varepsilon (\tau_\varepsilon(\partial G), \cdot) \notin
\partial^{\pm z} \right) + \sup_{v \in \p^{\pm z}} P_v ( \tau_\varepsilon >
\tau^*).
$$ We may now conclude the result by using Theorem \ref{teoescape0} and Remark
\ref{obsequivG}.
\end{proof}

With these two lemmas at hand, we can now show the remaning results of
\mbox{Section \ref{mainresults}.} Indeed, Theorem \ref{thm2} follows
from Theorem \ref{ecotsuplema1} by using the \mbox{strong Markov property} \mbox{together
with} Lemmas \ref{taubc} and \ref{nescapelema3}. Furthermore, Theorem \ref{thm3}
follows from Lemma \ref{taubc2} for $\rho = c$, where $c$ is as in Conditions
\ref{assumpg}, together with \eqref{ecota0} for $\delta > 0$ sufficiently small
so as to guarantee that $B_\delta(\pm z)$ is contained in the interior of
$B_{n_0}$, where $n_0$ is as in \eqref{predomain}. Finally, Lemma
\ref{nescapelema3} implies that $\lim_{\varepsilon \rightarrow 0}
\frac{\beta_\varepsilon}{\gamma_\varepsilon}=1$ from which, together with Lemma
\ref{taubc2} and the strong Markov property, we get Theorems \ref{thm4} and \ref{thm5} by using
Theorems \ref{asympteog} and \ref{stableg}. We leave the details to the reader,
which are completely straightforward.

\section{Appendix}

\subsection{Comparison principle}

\begin{prop} Let $f_1,f_2 : \R \to \R$ be globally Lipschitz functions. For $u,v \in
C([0,1])$ consider $U^u$ and $U^v$ the solutions of the equation
$$
\p_t U = \p^2_{xx} U + f_1(U) + f_2(U)\dot{W}
$$ with initial data $u$ and $v$, respectively, and boundary conditions
satisfying
$$
P( U(t,\cdot)|_{\p [0,1]} \geq V(t,\cdot)|_{\p [0,1]} \text{ for all $t \geq
0$}) = 1.
$$ Then, if $u \geq v$ we have that
$$
P\left( U^u(t,x) \geq U^v(t,x) \text{ for all }t \geq 0, x \in [0,1] \right)=1.
$$
\end{prop}
A proof of this result can be found on \cite[p.~130]{DNKXM}. Let us notice that
by taking $f_2 \equiv 0$ one obtains a comparison principle for deterministic
partial differential equations.

\subsection{Growth and regularity estimates}

\begin{prop}\label{G.1}Given a bounded set $B \subseteq C_D([0,1])$ there exists
$t_B > 0$ such that
\begin{itemize}
\item [$\bullet$] $\tau^u > t_B$ for any $u \in B$
\item [$\bullet$] There exists $b: [0,t_B] \rightarrow \R^+$ such that $\lim_{t
\rightarrow 0^+} b(t) = 0$ and for any $t \in [0,t_B]$
$$
\sup_{u \in B} \| U^u(t,\cdot) - u \|_\infty \leq b(t).
$$
\end{itemize}
\end{prop}

\begin{proof}
The first assertion follows at once from the comparison principle if one takes as $t_B$ any time for which the solution to the equation
\[
\left\{\begin{array}{l}\dot y = g(y) \\
y(0)=\sup_{u \in B} \|u\|_\infty\end{array}\right.
\]
is defined. For the second assertion, take $b$ as the solution to the equation
\[
\left\{\begin{array}{l} \dot y = g(y + \sup_{u \in B} \|u\|_\infty)
\\ y(0) = 0.\end{array}\right.
\]
and then apply the comparison principle again to obtain the desired inequality.
\end{proof}

\begin{prop}\label{G.2}The following local and pointwise growth estimates hold:
\begin{enumerate}
\item [i.] Given a bounded set $B \subseteq C_D([0,1])$ there exist $C_B, t_B >
0$ such that
\begin{itemize}
\item [$\bullet$] $\tau^u > t_B$ for any $u \in B$
\item [$\bullet$] For any pair $u,v \in B$ and $t \in [0,t_B]$
$$
\|U^u(t,\cdot) - U^v(t,\cdot) \|_\infty \leq e^{C_B t} \| u - v\|_\infty.
$$
\end{itemize}
\item [ii.] Given $u \in C_D([0,1])$ and $t \in [0,\tau^u)$ there exist
$C_{u,t}, \delta_{u,t} > 0$ such that
\begin{itemize}
\item [$\bullet$] $\tau^v > t$ for any $v \in B_{\delta_{u,t}}(u)$
\item [$\bullet$] For any $v \in B_{\delta_{u,t}}(u)$ and $s \in [0,t]$
$$
\|U^u(s,\cdot) - U^v(s,\cdot) \|_\infty \leq e^{C_{u,t} s} \|u - v\|_\infty.
$$
\end{itemize}
\end{enumerate}
\end{prop}

\begin{proof} These are standard continuity estimates with respect to the
initial datum which can be found, for example, in \cite{Pao}.
\end{proof}

\begin{prop}\label{A.1} If $u \in C_D([0,1])$ then $\partial^2_{xx} U^{u}$
exists for any $t \in (0,\tau^u)$. Furthermore, for any bounded set $B \subseteq
C_D([0,1])$ there exists a time $t_B > 0$ such that
\begin{enumerate}
\item [$\bullet$] $\tau^u > t_B$ for any $u \in B$
\item [$\bullet$] For any $t \in (0,t_B)$ we have $\sup_{u \in B} \left[\max\{
\| \partial_{x} U^{u} (t,\cdot) \|_{\infty}, \| \partial^2_{xx} U^{u} (t,\cdot)
\|_{\infty}\}\right] < + \infty$.
\end{enumerate}
\end{prop}

\begin{proof} One can obtain this result by following the analysis in the proof
of \cite[Lemma~A.1]{B1}.
\end{proof}

\begin{prop}\label{A.2} For any bounded set $B \subseteq C_D([0,1])$ there
exists $t_B > 0$ such that
\begin{enumerate}
\item [$\bullet$] $\tau^u > t_B$ for any $u \in B$
\item [$\bullet$] For any $t \in (0,t_B)$ there exist positive constants
$R_t,N_t$ such that for every $u \in B$ the function $U^{u}(t,\cdot)$ belongs to
the compact set
$$
\gamma_{R_t,N_t} = \{ v \in C_D([0,1]) : \|v\|_\infty \leq R_t \,,\, |v(x)-v(y)|
\leq N_t |x-y| \text{ for all }x,y \in [0,1]\}.
$$
\end{enumerate}
\end{prop}

\begin{proof} This is direct consequence of Propositions \ref{G.1}-\ref{A.1} and
the mean value theorem.
\end{proof}

\begin{prop}\label{A.7} The following local and pointwise growth estimates hold:
\begin{enumerate}
\item [i.] Given a bounded set $B \subseteq C_D([0,1])$ there exists $t_B > 0$
such that
\begin{itemize}
\item [$\bullet$] $\tau^u > t_B$ for any $u \in B$
\item [$\bullet$] For any $t \in (0,t_B)$ there exists $C_{t,B} > 0$ such that
for all $u,v \in B$
$$
\| \p_x U^{u}(t,\cdot) - \p_x U^v (t,\cdot) \|_\infty \leq C_{t,B} \| u -
v\|_\infty.
$$
\end{itemize}
\item [ii.] Given $u \in C_D([0,1])$ and $t \in (0,\tau^u)$ there exist
$C_{u,t}, \delta_{u,t} > 0$ such that
\begin{itemize}
\item [$\bullet$] $\tau^v > t$ for any $v \in B_{\delta_{u,t}}(u)$
\item [$\bullet$] For any $v \in B_{\delta_{u,t}}(u)$
$$
\| \p_x U^{u}(t,\cdot) - \p_x U^v (t,\cdot) \|_\infty \leq C_{u,t} \| u -
v\|_\infty.
$$
\end{itemize}
\end{enumerate}
\end{prop}

\begin{proof} These estimates also follow from the analysis in the proof
of \cite[Lemma~A.1]{B1}.
\end{proof}

\begin{prop}\label{A.3} For any equilibrium point $w$ of the deterministic
system let us consider its stable manifold $\mathcal{W}^w$ defined as
$$
\mathcal{W}^{w}:=\{ u \in C_D([0,1]) : U^{u} \text{ is globally defined and
}U^{u}(t,\cdot) \underset{t \rightarrow +\infty}{\longrightarrow} w\}.
$$ Notice that $\mathcal{W}^{\mathbf{0}}=\mathcal{D}_{\mathbf{0}}$. Then for any
bounded set $B \subseteq \mathcal{W}^w$ there exists $t_B > 0$ such that for any
$t_0 \in [0,t_B]$ and $r>0$ we have
$$
\sup_{u \in B} \left[ \inf \{ t \geq t_0 : d( U^u(t,\cdot), w) \leq r \} \right]
< +\infty
$$ whenever one of the following conditions hold:
\begin{enumerate}
\item [$\bullet$] $w \neq \mathbf{0}$
\item [$\bullet$] $w = \mathbf{0}$ and $B$ is at a positive distance from
$\mathcal{W} := \bigcup_{n \in \Z - \{0\}} \mathcal{W}^{z^{(n)}}$.
\end{enumerate} Furthermore, if $B \subseteq \mathcal{D}_e$ is bounded and at a
positive distance from $\mathcal{W}$ then for any $n \in \N$ we have
$$
\sup_{u \in B} \tau^{(n),u} < +\infty.
$$
\end{prop}

\begin{proof} Let us suppose first that $w \neq \mathbf{0}$. Then, since
$\mathcal{W}^\omega$ is a closed set, by Proposition \ref{A.2} we have that the
family $\{ U^u(t_B ,\cdot) : u \in B \}$ is contained in a compact set $B'
\subseteq \mathcal{W}^{w}$ for some suitably small $t_B > 0$. Hence, we obtain
that
$$
\sup_{u \in B} \left[ \inf \{ t \geq t_0 : d( U^u(t,\cdot), w) \leq r \} \right]
\leq t_B + \sup_{v \in B'} \left[ \inf \{ t \geq 0 : d( U^v(t,\cdot), w) < r \}
\right]
$$ Since the application $v \mapsto \inf \{ t \geq 0 : d( U^v(t,\cdot), w) < r
\}$ is upper semicontinuous and finite on $\mathcal{W}^w$, we conclude that the
right hand side is finite and thus the result follows in this case.
Now, if $w = \mathbf{0}$ then once again by Proposition \ref{A.2} we have that
the family $\{ U^u(t_B ,\cdot) : u \in B \}$ is contained in a compact set $B'
\subseteq \mathcal{D}_{\mathbf{0}}$. By Proposition \ref{G.1} we may choose $t_B
> 0$ sufficiently small so as to guarantee that $B'$ is at a positive distance
from $\mathcal{W}$. From here we conclude the proof as in the previous case.
Finally, the last statement of the proposition is proved in a completely
analogous fashion.
\end{proof}

\subsection{Properties of the potential $S$}

\begin{prop}\label{Lyapunov} The mapping $t \mapsto S( U^u(t,\cdot) )$ is
monotone decreasing and continuous for any $u \in H^1_0((0,1))$.
\end{prop}

\begin{proof} An easy computation shows that
\[
 \frac{d}{dt}S(U^u(t,\cdot))=  - \int_0^1 (\partial_t U^u(t,x))^2\,dx \leq 0
\]
from which the result follows. Details can be found in \cite[Lemma 17.5]{QS}.
\end{proof}

\begin{prop}\label{A.4} The potential $S$ is lower semicontinuous.
\end{prop}

\begin{proof} See \cite{E}.
\end{proof}

\begin{prop}\label{S.1} Given $u \in C_D([0,1])$ and $t \in (0,\tau^u)$ there
exist constants $C_{u,t}, \delta_{u,t} > 0$ such that
\begin{itemize}
\item [$\bullet$] $\tau^v > t$ for any $v \in B_{\delta_{u,t}}(u)$
\item [$\bullet$] For any $v \in B_{\delta_{u,t}}(u)$ one has
$$
\| S(U^{u}(t,\cdot)) - S(U^v (t,\cdot) ) \|_\infty \leq C_{u,t} \| u -
v\|_\infty.
$$
\end{itemize}
\end{prop}

\begin{proof} This is a direct consequence of Propositions \ref{A.7} and
\ref{G.2}.
\end{proof}

\subsection{Properties of the quasipotential $V$}

\begin{prop}\label{A.5} The mapping $u \mapsto V(\mathbf{0},u)$ is lower
semicontinuous on $C_D([0,1])$.
\end{prop}

\begin{proof} Let $(u_k)_{k \in \N} \subseteq C_D([0,1])$ be a sequence
converging to some limit $u_\infty \in C_D([0,1])$. We must check that
\begin{equation}\label{continferiorv}
V(\mathbf{0},u_\infty) \leq \liminf_{k \rightarrow +\infty} V(\mathbf{0},v_k).
\end{equation} If $S(u_\infty)=+\infty$ then by Proposition \ref{costo} we see
that $V(\mathbf{0},u_\infty)=+\infty$ and thus by the lower semicontinuity of
$S$ we conclude that $\lim_{v \rightarrow u} V(\mathbf{0},v)=+\infty$ which
establishes \eqref{continferiorv} in this particular case. Now, if $S(u_\infty)<
+\infty$ then, by the lower semicontinuity of $S$ and the continuity in time of
the solutions to \eqref{MainPDE}, given $\delta > 0$ there exists
$t_0 > 0$ sufficiently small such that $S(U^{u_\infty}(t_0,\cdot)) > S(u_\infty)
- \frac{\delta}{2}$. Moreover, by Proposition \ref{G.2} we may even assume that
$t_0$ is such that
$$
\| U^{u_k}(t_0,\cdot) - U^{u_\infty}(t_0,\cdot) \|_\infty \leq 2 \| u_k -
u_\infty\|_\infty
$$ for any $k \in \N$ sufficiently large. Thus, given $k$ sufficiently large and
a path $\varphi_k$ \mbox{from $\mathbf{0}$ to $u_k$} we construct a path
$\varphi_{k,\infty}$ from $\mathbf{0}$ to $u_\infty$ by the following steps:
\begin{enumerate}
\item [i.] We start from $\mathbf{0}$ and follow $\varphi_k$ until we reach
$u_k$.
\item [ii.] From $u_k$ we follow the deterministic flow $U^{u_k}$ until time
$t_0$.
\item [iii.] We then join $U^{u_k}(t_0,\cdot)$ and $U^{u_\infty}(t_0,\cdot)$ by
a linear interpolation of speed one.
\item [iv.] From $U^{u_\infty}(t_0,\cdot)$ we follow the reverse deterministic
flow until we reach $u_\infty$.
\end{enumerate}
By the considerations made in the proof of \cite[Lemma~4.3]{SJS} it is not
difficult to see that there exists $C > 0$ such that for any $k \in \N$
sufficiently large we have
$$
I(\varphi_{k,\infty}) \leq I(\varphi_k) + C \| u_k - u_\infty\|_\infty + \delta
$$ so that we ultimately obtain
$$
V(\mathbf{0},u_\infty) \leq \liminf_{k \rightarrow +\infty} V(\mathbf{0},u_k) +
\delta.
$$ Since $\delta > 0$ can be taken arbitrarily small we conclude
\eqref{continferiorv}.
\end{proof}

\begin{prop}\label{A.6} For any $u,v \in C_D([0,1])$ the map $t \mapsto
V\left(u,U^v(t,\cdot)\right)$ is decreasing.
\end{prop}

\begin{proof}Given $0 \leq s < t$ and a path $\varphi$ from $u$ to
$U^{u}(s,\cdot)$ we may extend $\phi$ to a path $\tilde{\varphi}$ from $u$ to
$U^{u}(t,\cdot)$ simply by following the deterministic flow afterwards.
It follows that
$$
V\left(u,U^v(t,\cdot)\right) \leq I(\tilde{\varphi}) = I(\varphi)
$$ which, by taking infimum over all paths from $u$ to $U^{u}(s,\cdot)$, yields
the \mbox{desired monotonicity.}
\end{proof}

{\bf Acknowledgments.} We thank Philippe Souplet for the discussions we had with him and the proof of Lemma \ref{compacidad}.

\bibliographystyle{plain}
\bibliography{bibliografia}

\def\cprime{$'$} \def\cprime{$'$}
\begin{thebibliography}{10}

\bibitem{B1}
Stella Brassesco.
\newblock Some results on small random perturbations of an infinite-dimensional
  dynamical system.
\newblock {\em Stochastic Process. Appl.}, 38(1):33--53, 1991.

\bibitem{B2}
Stella. Brassesco.
\newblock Unpredictability of an exit time.
\newblock {\em Stochastic Process. Appl.}, 63(1):55--65, 1996.

\bibitem{CGOV}
Marzio Cassandro, Antonio Galves, Enzo Olivieri, and Maria~Eul{\'a}lia Vares.
\newblock Metastable behavior of stochastic dynamics: a pathwise approach.
\newblock {\em J. Statist. Phys.}, 35(5-6):603--634, 1984.

\bibitem{CE1}
Carmen Cort{\'a}zar and Manuel Elgueta.
\newblock Large time behaviour of solutions of a nonlinear reaction-diffusion
  equation.
\newblock {\em Houston J. Math.}, 13(4):487--497, 1987.

\bibitem{CE2}
Carmen Cort{\'a}zar and Manuel Elgueta.
\newblock Unstability of the steady solution of a nonlinear reaction-diffusion
  equation.
\newblock {\em Houston J. Math.}, 17(2):149--155, 1991.

\bibitem{DNKXM}
Robert Dalang, Davar Khoshnevisan, Carl Mueller, David Nualart, and Yimin Xiao.
\newblock {\em A minicourse on stochastic partial differential equations},
  volume 1962 of {\em Lecture Notes in Mathematics}.
\newblock Springer-Verlag, Berlin, 2009.
\newblock Held at the University of Utah, Salt Lake City, UT, May 8--19, 2006,
  Edited by Khoshnevisan and Firas Rassoul-Agha.

\bibitem{E}
Lawrence~C. Evans.
\newblock {\em Partial differential equations}, volume~19 of {\em Graduate
  Studies in Mathematics}.
\newblock American Mathematical Society, Providence, RI, second edition, 2010.

\bibitem{FJL}
William~G. Faris and Giovanni Jona-Lasinio.
\newblock Large fluctuations for a nonlinear heat equation with noise.
\newblock {\em J. Phys. A}, 15(10):3025--3055, 1982.

\bibitem{FW}
Mark~I. Freidlin and Alexander~D. Wentzell.
\newblock {\em Random perturbations of dynamical systems}, volume 260 of {\em
  Grundlehren der Mathematischen Wissenschaften [Fundamental Principles of
  Mathematical Sciences]}.
\newblock Springer, Heidelberg, third edition, 2012.
\newblock Translated from the 1979 Russian original by Joseph Sz{\"u}cs.

\bibitem{GOV}
Antonio Galves, Enzo Olivieri, and Maria~Eul{\'a}lia Vares.
\newblock Metastability for a class of dynamical systems subject to small
  random perturbations.
\newblock {\em Ann. Probab.}, 15(4):1288--1305, 1987.

\bibitem{GS}
Pablo Groisman and Santiago Saglietti.
\newblock Small random perturbations of a dynamical system with blow-up.
\newblock {\em J. Math. Anal. Appl.}, 385(1):150--166, 2012.

\bibitem{MOS}
Fabio Martinelli, Enzo Olivieri, and Elisabetta Scoppola.
\newblock Small random perturbations of finite- and infinite-dimensional
  dynamical systems: unpredictability of exit times.
\newblock {\em J. Statist. Phys.}, 55(3-4):477--504, 1989.

\bibitem{IM}
Henry~P. McKean.
\newblock {\em Stochastic integrals}.
\newblock AMS Chelsea Publishing, Providence, RI, 2005.
\newblock Reprint of the 1969 edition, with errata.

\bibitem{M}
Carl Mueller.
\newblock Coupling and invariant measures for the heat equation with noise.
\newblock {\em Ann. Probab.}, 21(4):2189--2199, 1993.

\bibitem{OV}
Enzo Olivieri and Maria~Eul{\'a}lia Vares.
\newblock {\em Large deviations and metastability}, volume 100 of {\em
  Encyclopedia of Mathematics and its Applications}.
\newblock Cambridge University Press, Cambridge, 2005.

\bibitem{Pao}
C.~V. Pao.
\newblock {\em Nonlinear parabolic and elliptic equations}.
\newblock Plenum Press, New York, 1992.

\bibitem{QS}
Pavol Quittner and Philippe Souplet.
\newblock {\em Superlinear parabolic problems}.
\newblock Birkh\"auser Advanced Texts: Basler Lehrb\"ucher. [Birkh\"auser
  Advanced Texts: Basel Textbooks]. Birkh\"auser Verlag, Basel, 2007.
\newblock Blow-up, global existence and steady states.

\bibitem{SJS}
Santiago Saglietti.
\newblock {\em Metastability for a PDE with blow-up and the FFG dynamics in
  diluted models.}
\newblock PhD thesis.

\bibitem{S}
Philippe. Souplet.
\newblock Private communication.

\bibitem{SOW}
Richard~B. Sowers.
\newblock Large deviations for a reaction-diffusion equation with
  non-{G}aussian perturbations.
\newblock {\em Ann. Probab.}, 20(1):504--537, 1992.

\bibitem{W}
John~B. Walsh.
\newblock An introduction to stochastic partial differential equations.
\newblock In {\em \'{E}cole d'\'et\'e de probabilit\'es de {S}aint-{F}lour,
  {XIV}---1984}, volume 1180 of {\em Lecture Notes in Math.}, pages 265--439.
  Springer, Berlin, 1986.

\end{thebibliography}

\end{document}